\documentclass[11pt]{article}

\usepackage{amsfonts}
\usepackage{amsthm}
\usepackage{amsmath}
\usepackage{amssymb}
\usepackage{bm}
\usepackage{wasysym}

\usepackage{epsfig}
\usepackage[percent]{overpic}
\usepackage{graphics}
\usepackage{psfrag}
\usepackage[caption=false]{subfig}

\usepackage{titlesec}
\titleformat{\paragraph}
{\normalfont\normalsize\bfseries}{\theparagraph}{1em}{}
\titlespacing*{\paragraph}
{0pt}{3.25ex plus 1ex minus .2ex}{1.5ex plus .2ex}

\usepackage{verbatim}
\usepackage{color}
\usepackage{tikz}
\usetikzlibrary{arrows}
\usepackage{xcolor}
\definecolor{blun}{cmyk}{0.8, 0.5, 0, 0.7}

\usepackage{placeins}

\usepackage[all]{xy}
\usepackage{makeidx}
\usepackage{enumerate}
\usepackage[letterpaper,margin=0.9in]{geometry}
\usepackage[bookmarks,colorlinks,pdfauthor={KuehnSoresina},linkcolor=blun,citecolor=blun,urlcolor=blun]{hyperref}

\let\oldbibliography\thebibliography
\renewcommand{\thebibliography}[1]{%
  \oldbibliography{#1}%
  \setlength{\itemsep}{-1.2mm}%
}

\theoremstyle{plain}
\newtheorem{thm}{Theorem}[section]

\theoremstyle{definition}

\newtheoremstyle{myremark}
  {3pt}
  {3pt}
  {\small \rmfamily}
  {5pt}
  {\rmfamily}
  {:}
  {.5em}
  {}

\theoremstyle{myremark}

\def\txte{{\textnormal{e}}}

\def\tr{\textnormal{tr}}
\newcommand{\be}{\begin{equation}}
\newcommand{\ee}{\end{equation}}
\newcommand{\benn}{\begin{equation*}}
\newcommand{\eenn}{\end{equation*}}
\newcommand{\bea}{\begin{eqnarray}}
\newcommand{\eea}{\end{eqnarray}}
\newcommand{\beann}{\begin{eqnarray*}}
\newcommand{\eeann}{\end{eqnarray*}}

\newcommand{\myendex}{$\blacklozenge$\end{ex}}
\newcommand{\myendexerc}{$\lozenge$\end{exerc}}
\newcommand{\myendpexerc}{$\lozenge$\end{pexerc}}

\begin{document}

\author{
Christian Kuehn~\thanks{Department of Mathematics, Technical University of Munich, 85748 Garching b.~M\"unchen, Germany}
\thanks{Munich Data Science Institute, Technical University of Munich, 85748 Garching b.~M\"unchen, Germany}
\thanks{Complexity Science Hub Vienna, 1080 Vienna, Austria}~~and~
Cinzia Soresina \thanks{Department of Mathematics and Scientific Computing, University of Graz, 8010 Graz, Austria,\\ emails: \href{mailto: ckuehn@ma.tum.de}{ckuehn@ma.tum.de}, \href{mailto: cinzia.soresina@uni-graz.at}{cinzia.soresina@uni-graz.at}}}

\title{Cross-diffusion induced instability on networks}

\maketitle

\begin{abstract} \noindent
The concept of Turing instability,  namely that diffusion can destabilize the uniform steady state, is well known either in the context of partial differential equations (PDEs) or in networks of dynamical systems. Recently reaction--diffusion equations with cross-diffusion terms have been investigated, showing an analogous effect called cross-diffusion induced instability. In this paper, we extend this concept to networks of dynamical systems, showing that the spectrum of the graph Laplacian determines the instability appearance, as well as the spectrum of the Laplace operator in reaction--diffusion equations. We extend to network dynamics a particular network model for competing species, coming from the PDEs context. In particular, the influence of different topology structures on the cross-diffusion induced instability is highlighted, considering regular rings and lattices,  and also small-world, Erd\H{o}s--R\'eyni, and Barab\'asi--Albert networks.

\end{abstract}


\section{Introduction}\label{sec:intro}
A large number of real-life phenomena, for example in chemistry, ecology, and biology, give rise to a rich variety of complex behaviors, including pattern formation. The spirals that originate from chemical reactions, fish skin and animal coat patterning, and spatial vegetation patterns result from a spontaneous drive for self-organization into regular structures, both in time and space.
Mathematical principles that can drive the process of pattern formation have been established by Alan Turing in 1952. In his seminal paper on the theory of morphogenesis \cite{turing1952chemical}, he discovered that patterns can arise as a result of the dynamical interplay between reaction and diffusion. This theory, known as Turing instability, provides a general and elegant explanation for the variety of patterns appearing in living systems: diffusion perturbs and destabilizes a homogeneous stable equilibrium, yielding to spatially inhomogeneous steady states. 

On the other hand, Turing instability can also occur in networks of dynamical systems. In the seminal paper by Othmer and Scriven \cite{othmer1971instability}, a general mathematical framework for the analysis of instabilities in networks has been proposed and applied to regular lattices or small networks. Later it has been further explored and extended to more complex networks \cite{asllani2014theory,nakao2010turing}, and even later to multiplex~\cite{asllani2014turing,kouvaris2015pattern,brechtel2018master}, time-varying \cite{petit2017theory} and stochastic \cite{asslani2012stochastic} networks. Recently, Turing bifurcations have been investigated on one-dimensional random ring networks where the probability of a
connection between two nodes depends on the distance between them by using graphons \cite{bramburger2023pattern}. The surprising finding is that the conditions leading to Turing instability are the same for both, the reaction--diffusion and the network systems, but while in the continuum case, we look at the eigenvalues of the Laplace operator, in the network case we need the spectrum of the graph Laplacian. 

The graph Laplacian of an undirected network with $N$ nodes is a real, symmetric and positive semi-definite matrix, whose elements are given by
\begin{equation}\label{eq:L}
l_{ij}= k_i\delta_{ij}-a_{ij},
\end{equation}
where $a_{ij}$ are the elements of the adjacency matrix and $k_i$ is the degree of the node $i$, defined as
$$k_i=\sum_{i=1}^N{a_{ij}},\quad a_{ij}=\begin{cases}1,&i\leftrightarrow j,\\ 0&i\nleftrightarrow j, \end{cases} \quad i,j=1,\ldots,N.$$
In many areas and applications, it turns out that the eigenvalues of the graph Laplacian are useful tools, and for this reason, they have been intensively studied ~\cite{anderson1985eigenvalues, chung1997spectral, jost2022spectral,  li1998laplacian, merris1998note, mohar1991laplacian}. The eigenvalues $\Lambda_\alpha$ and eigenvectors $\underline{v}^{(\alpha)}=(v_1^{(\alpha)},\dots,v_N^{(\alpha)})^\top$ of the graph Laplacian $L$ are determined by 
$$\sum_{\alpha=1}^N{l_{ij}v_j^{(\alpha)}}=\Lambda_\alpha v_j^{(\alpha)}.$$
For an undirected network,  eigenvalues are real and nonnegative. It can be easily proven that the $0$ is an eigenvalue, $0= \lambda_1\leq \lambda_2\leq \dots \leq \lambda_N\leq N$, and that $\lambda_2>0$ if the network is connected. In particular, the eigenvalue $\lambda_2$ is often called \emph{algebraic connectivity}, and it holds that 
$$\lambda_2\leq \dfrac{2|E|}{N-1},$$ 
where $|E|$ is the number of edges, and $2|E|$ can be obtained by summing the diagonal elements of the graph Laplacian. Properties of the spectrum have been investigated for particular network structures~\cite{dorogovtsev2003spectra} including random graphs with given expected degrees~\cite{chung2004spectra} or more general random graphs~\cite{chung2011spectra}. Moreover, the spectrum of the graph Laplacian, as a quantity that encodes the topological structure of the underlying graph, also influences the dynamical properties, such as the stability of the synchronous state, which has been investigated for quite some time~\cite{OthmerScriven,SegelLevin} and has more recently been referred to as master stability function approach~\cite{bonacini2016single, MulasKuehnJost, PecoraCarroll, SunBolltNishikawa}.

The diffusion on the network given by the graph Laplacian only is a linear diffusion and considers pairwise interactions. More recently, nonlinear diffusion (also called nonlinear coupling or, in ecology, \textit{nonlinear dispersal}), and higher-order interactions on networks have been proposed and investigated, see for instance~\cite{bick2022multi, bonetto2022nonlinear,gambuzza2021stability, mancastroppa2023hyper}.

It is widely accepted that the Turing instability is induced by diffusion and requires an activator--inhibitor scheme of interaction between agents \cite{gierer1972theory}. However, in the context of continuous space reaction--diffusion systems a cross-diffusion model has been proposed to describe the spatial segregation of the species~\cite{shigesada1979spatial}. In this model, also known as Shigesada--Kawasaki--Teramoto (SKT) system, the reaction part does not generate an activator--inhibitor mechanism, but spatial patterns may arise thanks to cross-diffusion terms. These terms model those situations in which the movements of individuals of a species depend not only on the species itself but also on the presence of other species. This mechanism can also be observed in other contexts: for instance in the presence of chemotaxis, or also in epidemiology, where the movements of susceptible individuals are clearly influenced by the presence of infected ones~\cite{ottaviano2022global}. Then, cross-diffusion can destabilize a uniform equilibrium and induce pattern formation, even when it is not possible via standard diffusion terms; this phenomenon is known as cross-diffusion-induced instability. Reaction--cross-diffusion models have been extensively studied from different points of view and related to several applications (see~\cite{breden2021influence, conforto2018reaction, desvillettes2019non, gambino2008cross, kuehn2020numerical, lacitignola2018cross, soresina2022hopf, tang2015cross} and references therein).

Knowing all these ingredients, we are interested in the extension and the study of cross-diffusion and cross-diffusion-induced instability on complex networks. In addition to the natural pairing between continuous and discrete space models, this extension has been naively inspired by the discretization of reaction--cross-diffusion equations using finite differences \cite{daus2019entropic}. In fact, the (regular) mesh can be viewed as a regular lattice. Moreover, a natural question at this point is \emph{when/if} the discretized systems show the same patterns as the continuous model (depending on the number of mesh points).

While the theory of pattern formation on networks has been developed for several network structures and dynamical rules, in this direction only a few works have been investigated~\cite{zheng2017pattern,duan2019turing}, where the cross-diffusion terms considered are relatively simple (linear). We show that the theory of pattern formation on networks still holds considering general cross-diffusion terms (written in terms of the graph Laplacian). Also in this case, the conditions for cross-diffusion-driven instability are the same for the network and the continuous space case, where the eigenvalues of the graph Laplacian play the role of the eigenvalue of the Laplace operator. 
Then, we propose and investigate the SKT cross-diffusion model for competing species on a network. The model generalises the linear-diffusion one investigated in~\cite{slavik2020lotka}, which is known to have no heterogeneous (positive) steady states in the weak competition regime. As in the continuous setting, cross-diffusion is the key ingredient in producing non-homogeneous steady-state solutions. Our attention is however focused on the network structures (from regular rings, 2D-lattices to different complex/random graphs) in order to show, how they influence the possible dynamics of the system. In particular, we look at the spectrum of the graph Laplacian (or its distribution for random graphs). The aim of this work is mainly to point out that cross-diffusion terms can be useful ingredients in the study of complex systems, and that they can give rise to very rich dynamics depending crucially on the network topology. 
\medskip 

The paper is organized as follows. In Section \ref{sec:general}, we establish the general framework for cross-diffusion systems on networks and we extend the Turing instability analysis to cross-diffusion-induced instability. In Section \ref{netSKT} we propose a cross-diffusion model for competing species, inspired by the discretization of the SKT model on a 1D domain (presented in Appendix \ref{A:discrSKT}). This model is analyzed on different network structures, from regular rings and lattices, to more complex networks. The non-homogeneous steady states are studied, and the distributions of the eigenvalues for different network structures are presented. The focus of this section is to understand if particular structures are more favourable to leading to cross-diffusion-induced instability. Finally, in Section \ref{sec:conclusion} some concluding remarks can be found.  The link between the network model and the discretization of the reaction-diffusion SKT model is made in Appendix \ref{A:discrSKT}. The Python scripts for the simulations are freely accessible in the GitHub folder \cite{GitHubFolder}.

\section{The general framework}\label{sec:general}

In this section, we extend the Turing instability analysis on networks to cross-diffusion systems. This type of system can model complex natural phenomena in which nontrivial and nonlinear effects also affect the diffusion of the involved quantities. For instance, in ecology competition of species can cause migrations as a prey species tries to avoid predators, while in epidemiology susceptible individuals may try to avoid infected ones.

We consider an undirected network of~$N$ nodes. The network topology is encoded in the graph Laplacian~$L \in \mathbb{R}^{N\times N}$. On each node, we consider two state variables~$(\phi_i, \psi_i)=(\phi_i(t), \psi_i(t))$, evolving over time. The dynamics on each node is influenced by different components: 
\begin{itemize}
\item[-] the single node dynamics, described by the functions $f, \; g$,
\item[-] standard dispersal (random movements), expressed as a diffusive flux of a species to other nodes in terms of potential difference. The diffusion coefficients are denoted by $D_i,\, i=1,2$.
\item[-] cross-diffusion effects, due to the presence of a different species on other nodes. The cross-diffusion coefficients are denoted by $D_{ij},\, i,j=1,2$, while functions $c_i,\, i=1,2$ describes the inter-specific competition of the two species. 
\item[-] self-diffusion effects, due to the presence of the same species on other nodes. The self-diffusion coefficients are denoted by $D_{ii},\, i=1,2$, while functions $s_i,\, i=1,2$ describes the intra-specific competition among individual of the same species.
\end{itemize} 
All the considered functions are supposed to be regular enough (at least sufficiently smooth). 
Then the network dynamics can be modelled by
\begin{equation}\label{dsn_full}
\begin{cases}
\dot{\phi}_i=f(\phi_i,\psi_i)-D_1\displaystyle{ \sum_{j=1}^N l_{ij}\phi_j-D_{11}\sum_{j=1}^N l_{ij}s_1(\phi_j)\phi_j-D_{12}\sum_{j=1}^N l_{ij}c_1(\psi_j)\phi_j,}\\[0.3cm]
\dot{\psi}_i=g(\phi_i,\psi_i)-D_2 \displaystyle{\sum_{j=1}^N l_{ij}\psi_j-D_{22}\sum_{j=1}^N l_{ij}s_2(\psi_j)\psi_j-D_{21}\sum_{j=1}^N l_{ij}c_2(\phi_j)\psi_j,}\\[0.3cm]
i=1,\dots,N
\end{cases}
\end{equation}
where the overdot denotes time differentiation, and $l_{ij}$ are the elements of the graph Laplacian $L$ as defined in equation \eqref{eq:L} and they appear in the diffusion terms (standard, cross- and self-diffusion) since they encode the network structure. 

\medskip
As usual in Turing instability analysis, we consider a stable steady state for the single node dynamics $(\phi_*,\psi_*)$, such that 
$$f(\phi_*,\psi_*)=g(\phi_*,\psi_*)=0, \quad \tr({J_*})<0, \quad \det({J_*})>0,$$
where the matrix $J_*$ is the Jacobian of the single node dynamics evaluated at the homogeneous steady state
$$J_*=\begin{pmatrix}f_\phi(\phi_*,\psi_*)& f_\psi(\phi_*,\psi_*)\\[0.1cm] g_\phi(\phi_*,\psi_*)& g_\psi(\phi_*,\psi_*)\end{pmatrix}.$$ 
It turns out that the state
$$\phi_i=\phi_*, \quad \psi_i=\psi_*,\quad i=1,\dots,N,$$
is a homogeneous steady state for the network. Cross-diffusion induced instability arises when $(\phi_*,\psi_*)$ becomes unstable to inhomogeneous perturbations. 

The linear stability analysis is effectively a suitable combination of the standard cases on networks and in continuous media. We introduce small perturbations $\delta\phi_i,\, \delta\psi_i$ to the uniform state as 
$$(\phi_i, \psi_i)=(\phi_*,\psi_*)+(\delta\phi_i, \delta\psi_i),$$
and substitute this into equation \eqref{dsn_full}. Linearized differential equations for $\delta\phi_i,\, \delta\psi_i$ are
\begin{equation*}
\begin{cases}
\dot{\delta\phi}_i=f_\phi(\phi_*,\psi_*)\delta\phi_i+f_\psi(\phi_*,\psi_*)\delta\psi_i
-D_1\displaystyle{ \sum_{j=1}^N l_{ij}\delta\phi_j}\\
\qquad
-\displaystyle{D_{11}\sum_{j=1}^N l_{ij}\left( s_1(\phi_*)+s'_1(\phi_*)\phi_* \right)\delta\phi_j 
-D_{12}\sum_{j=1}^N l_{ij}\left( c_1(\psi_*)\delta\phi_j+ c'_1(\psi_*)\phi_*\delta\psi_j\right),}\\[0.3cm]

\dot{\delta\psi}_i=g_\phi(\phi_*,\psi_*)\delta\phi_i+g_\psi(\phi_*,\psi_*)\delta\psi_i
-D_2\displaystyle{ \sum_{j=1}^N l_{ij}\delta\psi_j}\\
\qquad
-\displaystyle{D_{22}\sum_{j=1}^N l_{ij}\left( s_2(\psi_*)+s'_2(\psi_*)\psi_* \right)\delta\psi_j 
-D_{21}\sum_{j=1}^N l_{ij}\left( c'_2(\phi_*)\psi_*\delta\phi_j+c_2(\phi_*)\delta\psi_j \right),}\\[0.3cm]
i=1,\dots,N,
\end{cases}
\end{equation*}
and they can be written as
\begin{equation}\label{eq:ODEperturbations}
\begin{pmatrix} \dot{\delta\phi}_i\\\dot{\delta\psi}_i\end{pmatrix}=
\displaystyle{J_*\begin{pmatrix} \delta\phi_i\\\delta\psi_i\end{pmatrix} + D_*\sum_{j=1}^N  l_{ij}}
\begin{pmatrix} \delta\phi_j\\\delta\psi_j\end{pmatrix},\quad i=1,\dots,N,
\end{equation}
where the matrix $D_*$ is the linearization of the diffusion part evaluated at the uniform steady state, given by
$$D_*=\begin{pmatrix}D_1+D_{11}\left( s_1(\phi_*)+s'_1(\phi_*)\phi_* \right)+D_{12}c_1(\psi_*)& D_{12}c'_1(\psi_*)\phi_*\\ 
D_{21}c'_2(\phi_*)\psi_*& D_2+D_{22}\left(  s_2(\psi_*)+s'_2(\psi_*)\psi_*\right) + D_{21}c_2(\phi_*)\end{pmatrix}.$$
We consider the spectrum of the graph Laplacian
$$\sum_{\alpha=1}^N {l_{ij}v_j^{(\alpha)}}=\Lambda_\alpha v_i^{(\alpha)}, \quad i,\alpha=1,\dots,N,$$
where $\Lambda_\alpha$ and $v^{(\alpha)}$ represent the eigenvalues and their associated eigenvectors, respectively. By expanding the perturbations $\delta\phi_i,\, \delta\psi_i$ over the set of Laplacian eigenvectors $v_i^{(\alpha)}$ as
$$\delta\phi_i(t)=\sum_{\alpha=1}^N {c_\alpha \txte^{\lambda_\alpha t} v_i^{(\alpha)}},\qquad 
\delta\psi_i(t)=\sum_{\alpha=1}^N {b_\alpha\txte^{\lambda_\alpha t} v_i^{(\alpha)}},$$
where the constants $c_\alpha$ and $b_\alpha$ refer to the initial conditions,  system \eqref{eq:ODEperturbations} is transformed into $N$ independent linear equations for different normal modes, resulting in the following eigenvalue equation for each $\alpha$:
\begin{equation*}
\lambda_\alpha \begin{pmatrix} c_\alpha\\b_\alpha\end{pmatrix}= (J^*-\Lambda_\alpha D_*)\begin{pmatrix} c_\alpha\\b_\alpha\end{pmatrix}, \qquad \alpha=1,\dots,N.
\end{equation*}
The matrix 
$$M^*_\alpha=J^*-\Lambda_\alpha D_*, \qquad \alpha=1,\dots,N$$
is called characteristic matrix. Then, in order to destabilize the homogeneous steady state, the characteristic matrix $M^*_\alpha$ must have a positive eigenvalue for at least one $\alpha$ corresponding to one eigenvalue $\Lambda_\alpha$ of the graph Laplacian. Note that the eigenvalues of the graph Laplacian appear only in combination with diffusion coefficients.\medskip

\noindent
\textbf{Remark.} Since the eigenvalues of the graph Laplacian are nonnegative, the expression of the characteristic matrix $M^*_\alpha$ turns out to be the same as in the PDEs setting, and the eigenvalues of the graph Laplacian play the role of the eigenvalues of the Laplace operator. Then the conditions on the parameters leading to the instability of the homogeneous steady state are the same.\medskip

In summary, the Turing instability analysis in the continuous setting and on discrete models can be adapted to cross-diffusion systems on networks. Note that depending on the single node dynamics, it is not always possible to destabilize the homogeneous equilibrium by means of standard diffusion terms only (in particular when the system does not have the activator--inhibitor structure). Cross-diffusion induced instability appears when the homogeneous steady state becomes unstable to inhomogeneous perturbations, due to the additional presence of cross-diffusion terms.\medskip 

\noindent
\textbf{Remark.} 
When $\Lambda_\alpha=0$ the characteristic matrix \eqref{eq:Mk} reduces to the Jacobian $J_*$ of the single node dynamics evaluated at the coexistence steady state. Therefore the zero-eigenvalue provides information about the dynamics of an isolated node. The non-zero eigenvalues instead account for the more complex behavior of the network.

\section{The SKT model on networks}\label{netSKT}
The goal of this section is to show that the cross-diffusion terms can destabilize the uniform equilibrium state leading to non-trivial steady states and to study the role and the influence of the network topology on these states. To achieve this, we present here a cross-diffusion model on networks for competing species, analogous to the SKT model presented in the framework of PDEs (see~\cite{breden2021influence, kuehn2020numerical} and references therein). We consider two species, $u$ and $v$, competing for the same resource. A system of ODEs describes the dynamics of an isolated node
\begin{equation}\label{eq:singlenode}
\begin{cases}
\dot{u}=f(u,v)=r_1u-a_1 u^2-b_1uv,\\[0.3cm]
\dot{v}=g(u,v)=r_2v-b_2 uv-a_2v^2,\\
\end{cases}
\end{equation}
where $r_1,\,r_2$ are the growth rates, $a_1,\, a_2$ the intra-specific competition rates and $b_1,\,b_2$ the inter-specific competition rates. It is simple to show that the outcomes of \eqref{eq:singlenode} can be the total extinction, competitive exclusion or the coexistence of the species, depending on the parameter values and on the initial conditions. 

We consider now a network of $N$ nodes, its topology being fixed by the Laplacian matrix of the graph~$L=(l_{ij})_{i,j=1,\dots,N}$. On each node the population sizes are denoted with $u_i,\,v_i,\, i=1,\dots,N$ and the dynamics on each node in case of isolation (neglecting the diffusion process) is given by \eqref{eq:singlenode}. When the nodes are connected in a network, we consider diffusive movements of the individuals between linked nodes modelled in the usual way
$$-\sum_{j=1}^N l_{ij}u_j=\displaystyle{\sum_{j=1}^N (u_j-u_i)}.$$
Then we consider cross- and self-diffusion effects, which cause extra movements of individuals due to intra- and inter-competition pressure. In particular, individuals living on a node leave to reach other nodes because of the presence of individuals of the same species (self-diffusion) or of the competing species (cross-diffusion). Figure~\ref{fig:gain_loss_terms} summarizes the gain and loss terms considered here. 
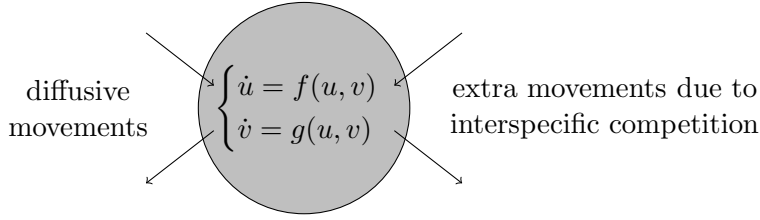
\begin{figure}
\begin{center}
\begin{tikzpicture}
\filldraw[fill=gray!50!white] (0,0) circle (40pt);
\draw[](0.1,0) node {$\begin{cases}
    \dot{u}=f(u,v)\\ \dot{v}=g(u,v)\\
\end{cases}$};

\draw [->] (-1.2,-0.3) -- (-2.1,-1);

\draw [->] (-2.1,1) -- (-1.2,0.3);
\node[align=center] at (-3,0) {diffusive\\ movements};

\draw [->] (1.2,-0.3) -- (2.1,-1);

\draw [->] (2.1,1) -- (1.2,0.3);
\node[align=center] at (4,0) {extra movements due to\\ interspecific competition};
\end{tikzpicture}
\end{center}
\caption{Schematic representation of the dynamics on a node in the network.}\label{fig:gain_loss_terms}
\end{figure}

In detail, they can be described by
$$-\sum_{j=1}^N l_{ij}u_j^2,\qquad -\sum_{j=1}^N l_{ij}v_ju_j.$$
Then, the system of ODEs describing the dynamics of the network is given by
\begin{equation}\label{eq:sysnet}
\begin{cases}
\dot{u}_i=f(u_i,v_i)-d_1\displaystyle{ \sum_{j=1}^N l_{ij}u_j-D_{11}\sum_{j=1}^N l_{ij}u_j^2-D_{12}\sum_{j=1}^N l_{ij}v_ju_j},\\[0.3cm]
\dot{v}_i=g(u_i,v_i)-d_2 \displaystyle{\sum_{j=1}^N l_{ij}v_j-D_{22}\sum_{j=1}^N l_{ij}v_j^2-D_{21}\sum_{j=1}^N l_{ij}u_jv_j},\\[0.3cm]
i=1,\dots,N
\end{cases}
\end{equation}
where $d_i,\, D_{ij}, \, i,j=1,2$ are the coupling strengths (or diffusion coefficients).

In order to obtain cross-diffusion instability effects, it is sufficient to consider $D_{11}=D_{22}=0$ and $d_1=d_2=d$ (analogous to the PDEs case). This choice corresponds to the simpler cross-diffusion system
\begin{equation}\label{eq:sysnet_noself}
\begin{cases}
\dot{u}_i=f(u_i,v_i)-d\displaystyle{ \sum_{j=1}^N l_{ij}u_j-D_{12}\sum_{j=1}^N l_{ij}v_ju_j},\\[0.3cm]
\dot{v}_i=g(u_i,v_i)-d\displaystyle{\sum_{j=1}^N l_{ij}v_j-D_{21}\sum_{j=1}^N l_{ij}u_jv_j}.\\[0.3cm]
i=1,\dots,N.
\end{cases}
\end{equation}

Following \cite{slavik2020lotka}, it can be proven that solutions with nonnegative initial conditions remain nonnegative all the time.
\begin{thm} If $u,\,v: \mathbb{R}_+\to \mathbb{R}^N$ satisfy \eqref{eq:sysnet_noself} and $u(0),\, v(0)\geq 0$, then $u(t),\,v(t)\geq 0$ for all $t>0$.
\end{thm}

\begin{proof}
We must prove the fact that the set
$$S=\{(u^1,u^2) \in \mathbb{R}^N \times \mathbb{R}^N : u^1\geq 0, \, v^1\geq 0\}$$
is a positively invariant region for system \eqref{eq:sysnet_noself}. The calculations follow \cite{slavik2020lotka} and it is easy to see that the extra terms coming from cross-diffusion have the correct sign configuration. 
\end{proof}

\textbf{Remark.} When only linear diffusion is considered \cite{slavik2020lotka}, a priori bounds for the solutions can be obtained. With cross-diffusion terms, we cannot obtain comparison results as in \cite[Theorem 3.2]{slavik2020lotka}.\medskip

We now proceed with the linearized analysis close to the coexistence state $(u_*,v_*)$ in the weak-competition regime (namely $a_1a_2-b_1b_2>0$ and $(u_*,v_*)$ is stable for the single node dynamics \eqref{eq:singlenode}). Then, according to Section \ref{sec:general}, the characteristic matrix is
\begin{equation}\label{eq:Mk}
M_\Lambda=J_*-\Lambda D_*
\begin{pmatrix}
-a_1 u_* & -b_1 u_*\\ 
-b_2 v_* & -a_2 v_*
\end{pmatrix}
-\Lambda 
\begin{pmatrix}
d+D_{12}v_*& D_{12}u_*\\ 
D_{21}v_*& d+D_{21}u_*
\end{pmatrix},
\end{equation}
where $\Lambda$ denotes an eigenvalue of the graph Laplacian. Remembering that the graph Laplacian has non-negative eigenvalues, then the characteristic matrix has a negative trace and its determinant determines the stability/instability of the homogeneous steady state in which $u_i=u_*,\; v_i=v_*,\; i=1,\dots, N$. In particular, we consider the \emph{weak competition case} $a_1a_2-b_1b_2>0$, leading to a stable steady state where $\tr (J_*)<0$ and $\det(J_*)>0$. Note that the obtained characteristic matrix is the same as in the reaction--cross-diffusion system and the difference between the continuum and the discrete case is played by the eigenvalues of the Laplacian and the graph Laplacian respectively. The trace of the characteristic matrix is negative. Then a cross-diffusion induced instability may appear if its determinant becomes negative for at least one $\Lambda$. The determinant can be written as
\begin{equation}\label{eq:detM_d}
\det (M_\Lambda)=A_\Lambda d^2+B_\Lambda d+C_\Lambda,
\end{equation}
where the coefficients $A_\Lambda,\;B_\Lambda,\; C_\Lambda$ of the second order polynomial in $d$ are given by
$$A_\Lambda=\Lambda^2,\quad 
B_\Lambda=D_{12}v_*\Lambda^2+D_{21}u_*\Lambda^2-\tr (J_*\Lambda), \quad 
C_\Lambda=-(D_{12}\alpha+D_{21}\beta)\Lambda+\det({J_*}).$$
In order to have a negative determinant, we need $C_\Lambda<0$. Note that there are two cases:
\begin{itemize}
\item $D_{12}\alpha+D_{21}\beta\leq0$, then no bifurcation can appear, independently of the network topology.
\item $D_{12}\alpha+D_{21}\beta>0$, then $C_\Lambda$ can be negative and bifurcations can appear depending on the eigenvalues of the Laplacian of the graph. In detail, we obtain the following threshold 
$$\Lambda>\Lambda_*=\dfrac{\det({J_*})}{D_{12}\alpha+D_{21}\beta}.$$
\end{itemize}

\noindent
\textbf{Remark.} This condition can be combined with other information on the spectrum of a particular graph to prove the stability of the homogeneous state. Since $\Lambda_N\leq N$, the simplest observation is that if $N<\Lambda_*$, then no cross-diffusion induced instability is possible. We can also obtain conditions regarding the stability of a particular mode (eigenvalue). For instance, for a regular graph of $N$ nodes and degree $2K$, we know that the first nontrivial eigenvalue (also known as \emph{algebraic connectivity}) satisfies $\Lambda_2\leq 2K N / (N-1)$. If $\Lambda_*>2K N / (N-1)$, then the mode related to $\Lambda_2$ is stable.

\medskip
\noindent
\textbf{Remark.} From equation \eqref{eq:detM_d} it is easy to see the cross-diffusion terms are the key ingredient to the appearance of nonhomogeneous steady states. In fact, if $D_{12}=D_{21}=0$, then $\det (M_\Lambda)$. This means that, regardless of the network structure, standard diffusion terms cannot lead to Turing instability.

\medskip
We look now at the determinant of the characteristic matrix as a second-order polynomial in $\Lambda$:
$$\det (M_\Lambda (\Lambda))= 
d(d+D_{12}v_*+D_{21}u_*)\Lambda^2
-(D_{12}\alpha+D_{21}\beta+d~\tr (J_*))\Lambda
+\det{J_*}.$$
Depending on the parameter set (the discriminant of the second-order polynomial must be positive), we can obtain an instability region $\Omega_*=(\Lambda_{*1},\Lambda_{*2})$ for the eigenvalues of the graph Laplacian. 
\subsection{Cross-diffusion induced instability in the SKT network}
This section is mainly devoted to showing that cross-diffusion terms are able to destabilize the uniform state, leading to the so-called cross-diffusion induced instability. The key point is that this effect is not possible without them. To this end, we use a simple network topology, in order to highlight the effect per se. The set of fixed parameters appears in Table~\ref{tab:SKTparam}, while the others will be specified each time in the text. The different network structures have been generated using the Python \cite{python} package NetworkX \cite{hagberg2008exploring}.

\textbf{Remark.} In the simulations, we do not vary the linear diffusion parameter $d$. As in the PDEs case, large values of $d$ tend to stabilise the homogeneous state.

\begin{table}
\centering
\begin{tabular}{ccccccccc}
\hline\\[-0.3cm]
$r_1$&$r_2$&$a_1$&$a_2$&$b_1$&$b_2$&$d$&$d_{12}$&$d_{21}$\\
\hline\\[-0.3cm]
5&2&3&3&1&1&0.03&3&0\\
\hline
\end{tabular}
\caption{Set of parameter values of the single node dynamics and diffusion coefficients of the SKT model~\eqref{eq:sysnet_noself} used in the numerical simulations. The set~$r_i,\,a_i,\,b_i,\, (i=1,2)$ corresponds to the weak competition case~($a_1a_2-b_1b_2>0$), namely the homogeneous steady states is stable for the single node dynamics. The remaining parameters related to the network structure will be specified in the text.}
\label{tab:SKTparam}
\end{table}

\subsubsection{$2K$-regular ring lattice}
A~$2K$-regular ring lattice is a graph with~$N$ nodes in a ring structure in which each node is connected to its~$2K$ neighbors~($K$ on either side)~\cite{wu2011robustness}. The associated graph Laplacian is a matrix with three bands (in the centre and in the corners NE and SW), defined by
$$l_{ij}=\begin{cases}
2K, & i=j,\\
-1,& 0<|i-j| \textnormal{mod}(N-1-K)\leq K,\\
0, & j<i-K \textnormal{ or } j>i+K.\end{cases}$$
The structures of the network and of the graph Laplacian are shown in Figure~\ref{fig:gLap_ring2K}.
The closed formula for the eigenvalues is 
$$\Lambda_j=2K-\sum_{k=1}^{K}{2\cos{\left( \dfrac{2\pi k(j-1)}{N}\right)}}, \quad j=1,\dots,N.$$
\begin{figure}
\subfloat[\label{ring_net}]{
\begin{overpic}[width=0.5\textwidth, trim=8cm 1.5cm 8cm 2cm, clip]{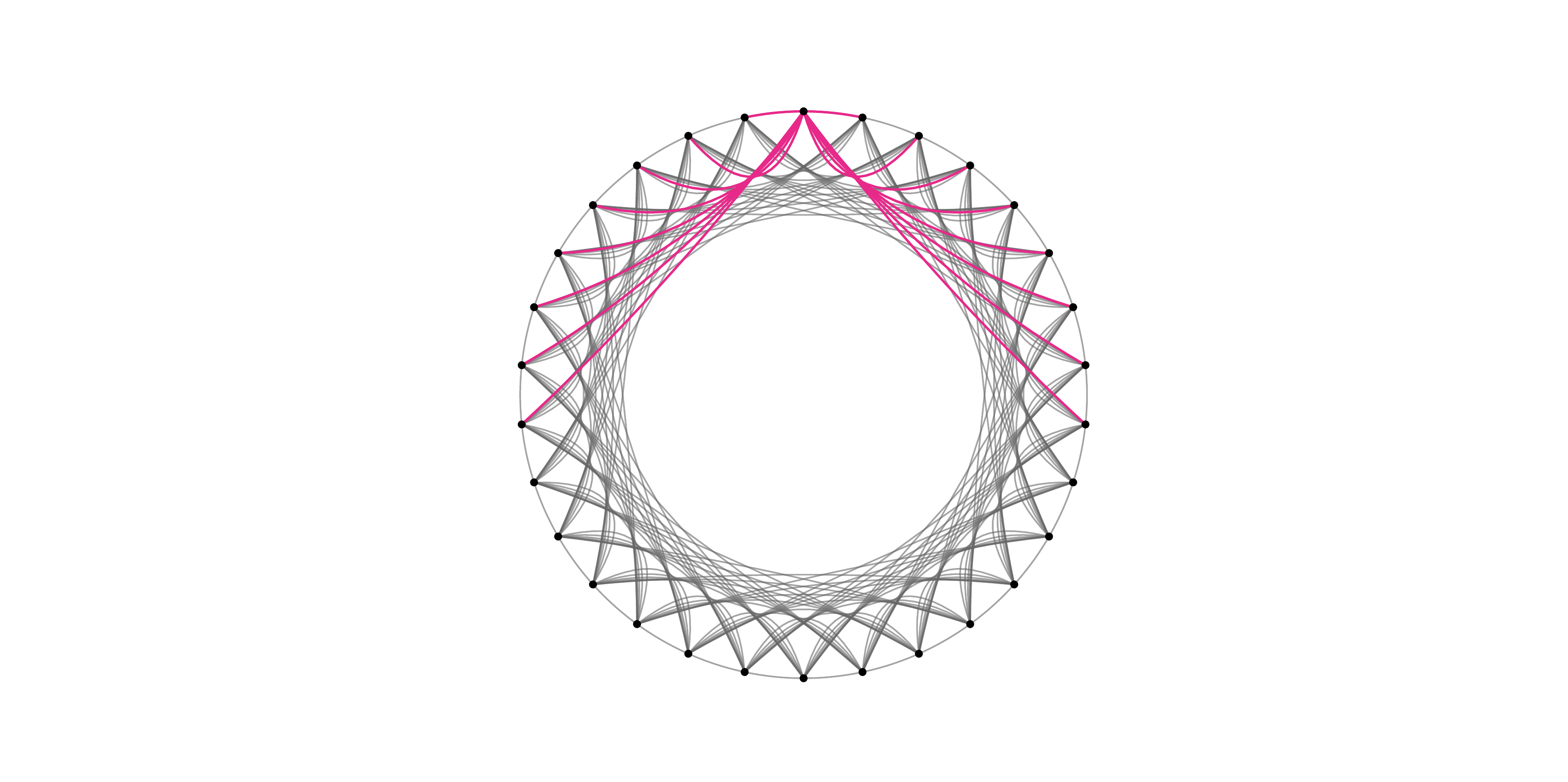}
\end{overpic}
}
\subfloat[\label{ring_Lap}]{
\begin{overpic}[width=0.5\textwidth]{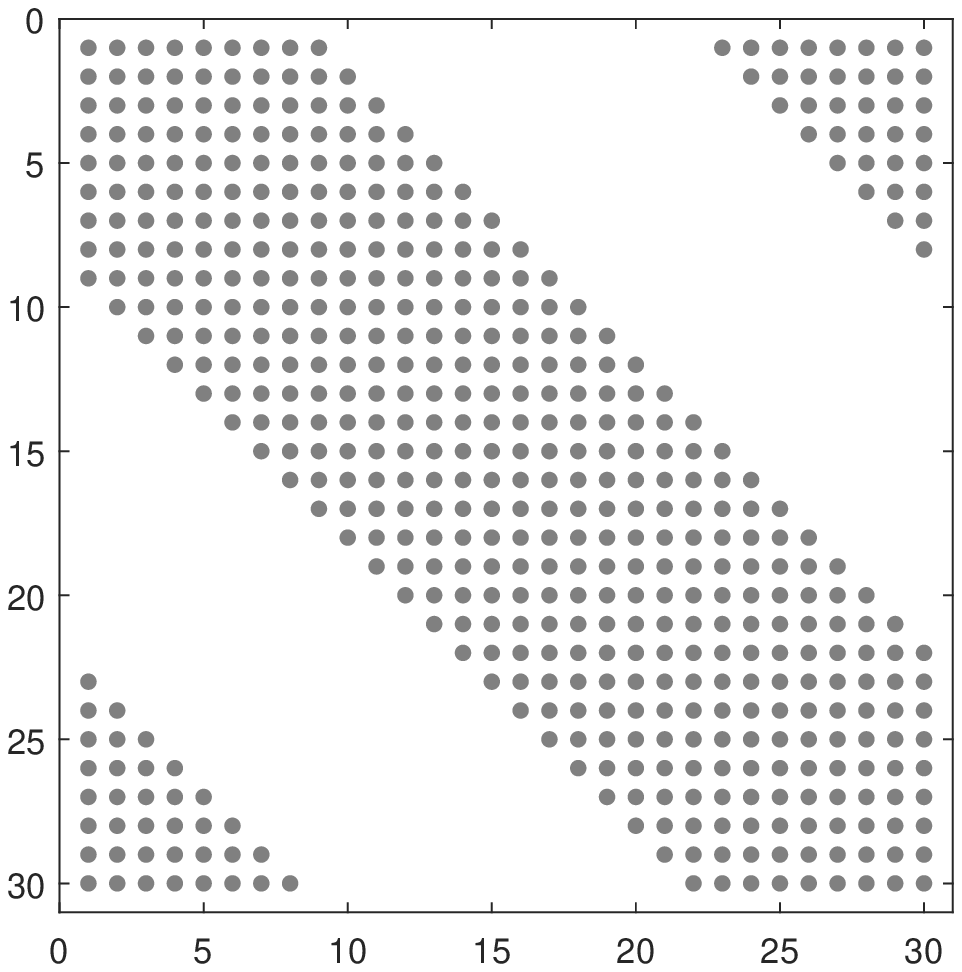}
\end{overpic}
}
\caption{Structure of a $2K$-regular ring lattice and of its graph Laplacian. \protect\subref{ring_net} Structure of the regular $2K$-ring with $N=30,\, K=8$, in which the $2K$ nearest neighbors ($K$ on each side) of a particular node are highlighted in magenta. \protect\subref{ring_Lap} Structure of the corresponding graph Laplacian, where grey dots indicate non-zero elements.}\label{fig:gLap_ring2K}
\end{figure}

In Figure \ref{fig:ring_spectrum_NK} the spectrum of the graph Laplacian is reported for different values of $N$ and $K$, in order to study the possible appearance of patterns. In particular, if an eigenvalue is located in the instability region $\Omega_*=(\Lambda_{*1},\Lambda_{*2})$, marked with a red stripe, then system \eqref{eq:sysnet_noself} admits a stable nonhomogenous steady state.
We can observe in Figure \ref{ring_spectrum_K} that the value of nonzero eigenvalues increases when $K$ increases (and $N$ is fixed); for small values of $K$ the spectrum lies below the threshold value $\Lambda_{*1}$ while increasing $K$ we pass from a situation in which part of the spectrum is located in the instability region to a situation in which only the smallest nonzero eigenvalue is present. Finally, for $K>25$ all the nonzero eigenvalues are greater than the threshold value $\Lambda_{*2}$, namely no stable pattern can appear. On the contrary, increasing $N$ with a fixed $K$ leads to smaller eigenvalues in the spectrum intersecting the instability region (Figure \ref{ring_spectrum_N}). However, this will lead to different nonhomogeneous solutions. 
\begin{figure}
\subfloat[\label{ring_spectrum_K}]{
\begin{overpic}[width=0.5\textwidth]{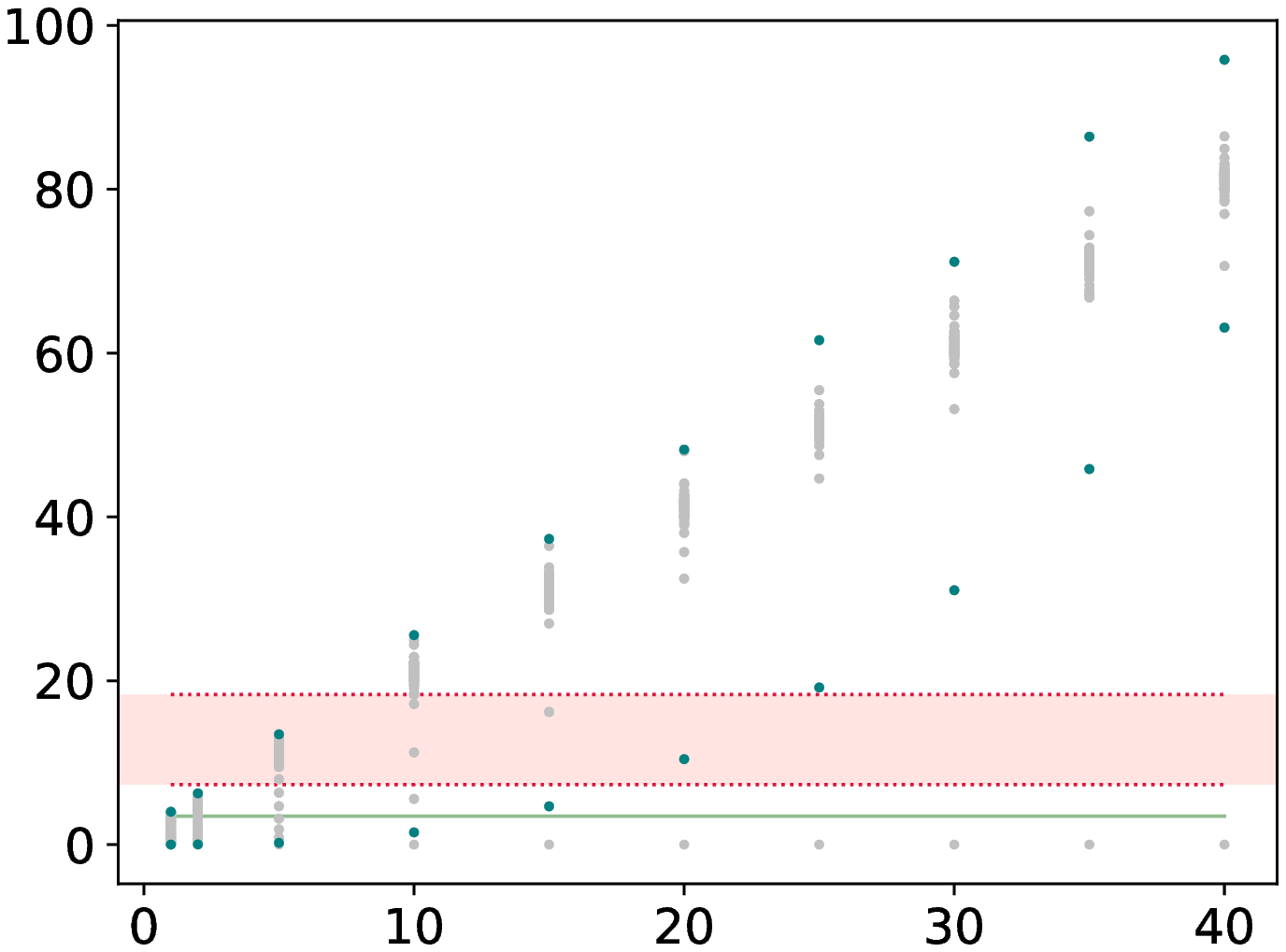}
\put(75,0){$K$}	
\put(92,10){$\Lambda_{*}$}	
\put(92,15){$\Lambda_{*1}$}	
\put(92,20){$\Lambda_{*2}$}	
\put(2,35){$\Lambda$}	
\end{overpic}
}
\subfloat[\label{ring_spectrum_N}]{
\begin{overpic}[width=0.5\textwidth]{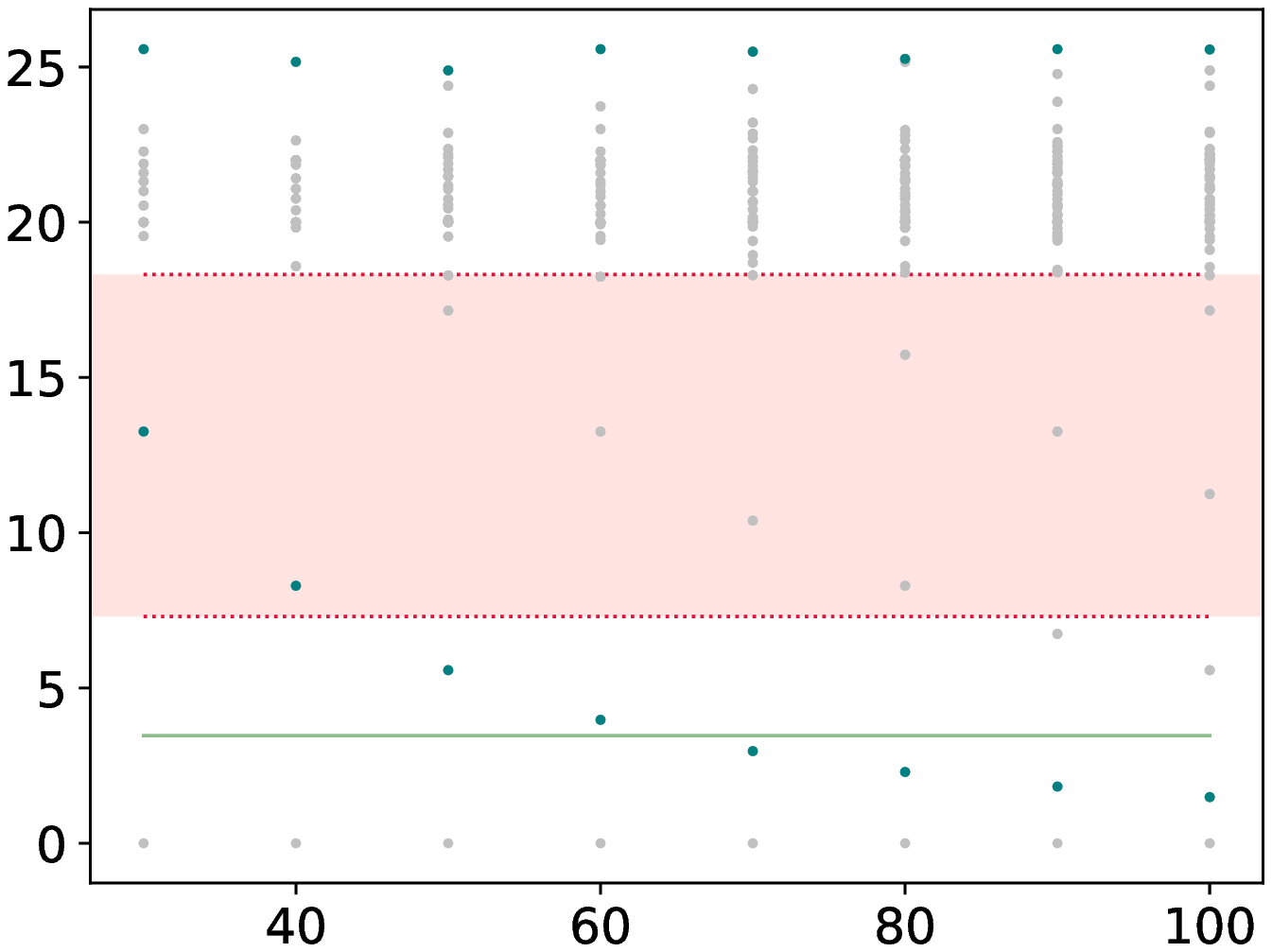}
\put(75,0){$N$}	
\put(92,15){$\Lambda_{*}$}	
\put(92,25){$\Lambda_{*1}$}	
\put(92,47){$\Lambda_{*2}$}	
\put(2,35){$\Lambda$}	
\end{overpic}
}
\caption{Cross-diffusion driven instability on regular rings. Grey dots mark the eigenvalues, while the first and the last nonzero eigenvalues appear in green. The instability region $\Omega_*=(\Lambda_{*1},\Lambda_{*2})$ and the threshold $\Lambda_*$ are marked with the red stripe and the green horizontal solid line, respectively (relevant to the parameter set in Table~\ref{tab:SKTparam} and $d=0.03$). \protect\subref{ring_spectrum_K} Spectrum of the graph Laplacian of the regular $2K$-ring with $N=100$ and different values of $K$. \protect\subref{ring_spectrum_N} Spectrum of the graph Laplacian of the regular $2K$-ring with $K=10$ and different values of $N$.}\label{fig:ring_spectrum_NK}
\end{figure}

In Figure~\ref{fig:ring_dyn_NK} we show different outcomes of the cross-diffusion model \eqref{eq:sysnet_noself} for the ring structure of $N=100$ nodes and three different values of $K$, corresponding to different locations of the spectrum with respect to the instability region. Trajectories of each node over time and the final configuration of the network are reported: varying $K$ the final configuration is not homogeneous, but a different pattern can appear. 
\begin{figure}
\subfloat[$K=10$\label{RING_N100_K10}]{
\begin{overpic}[width=0.5\textwidth]{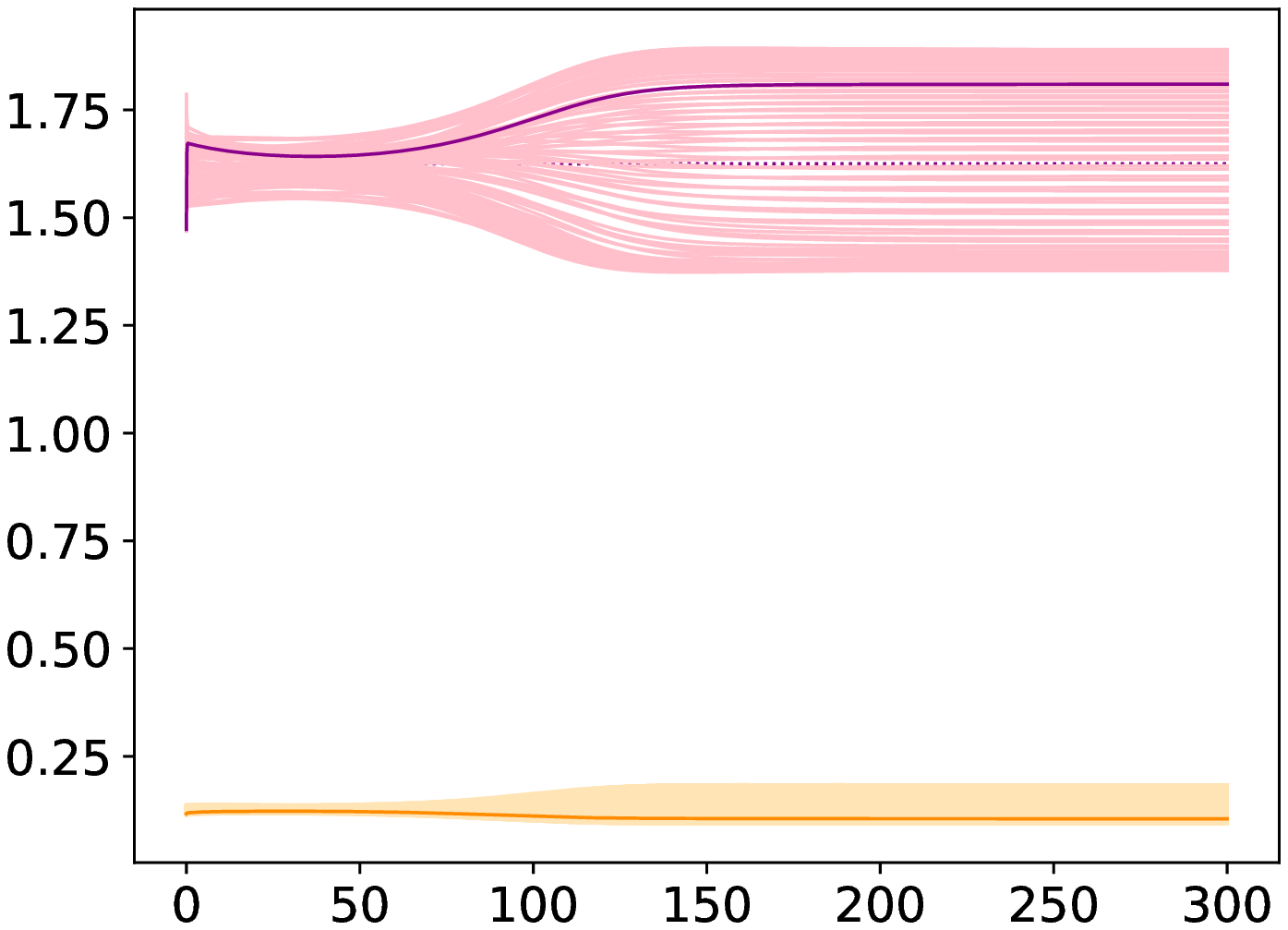}
\put(50,-2){$t$}
\put(0,35){\rotatebox{90}{$u_i,\,v_i$}}		
\put(92,55){$u_*$}	
\put(92,10){$v_*$}
\end{overpic}
\begin{overpic}[width=0.5\textwidth]{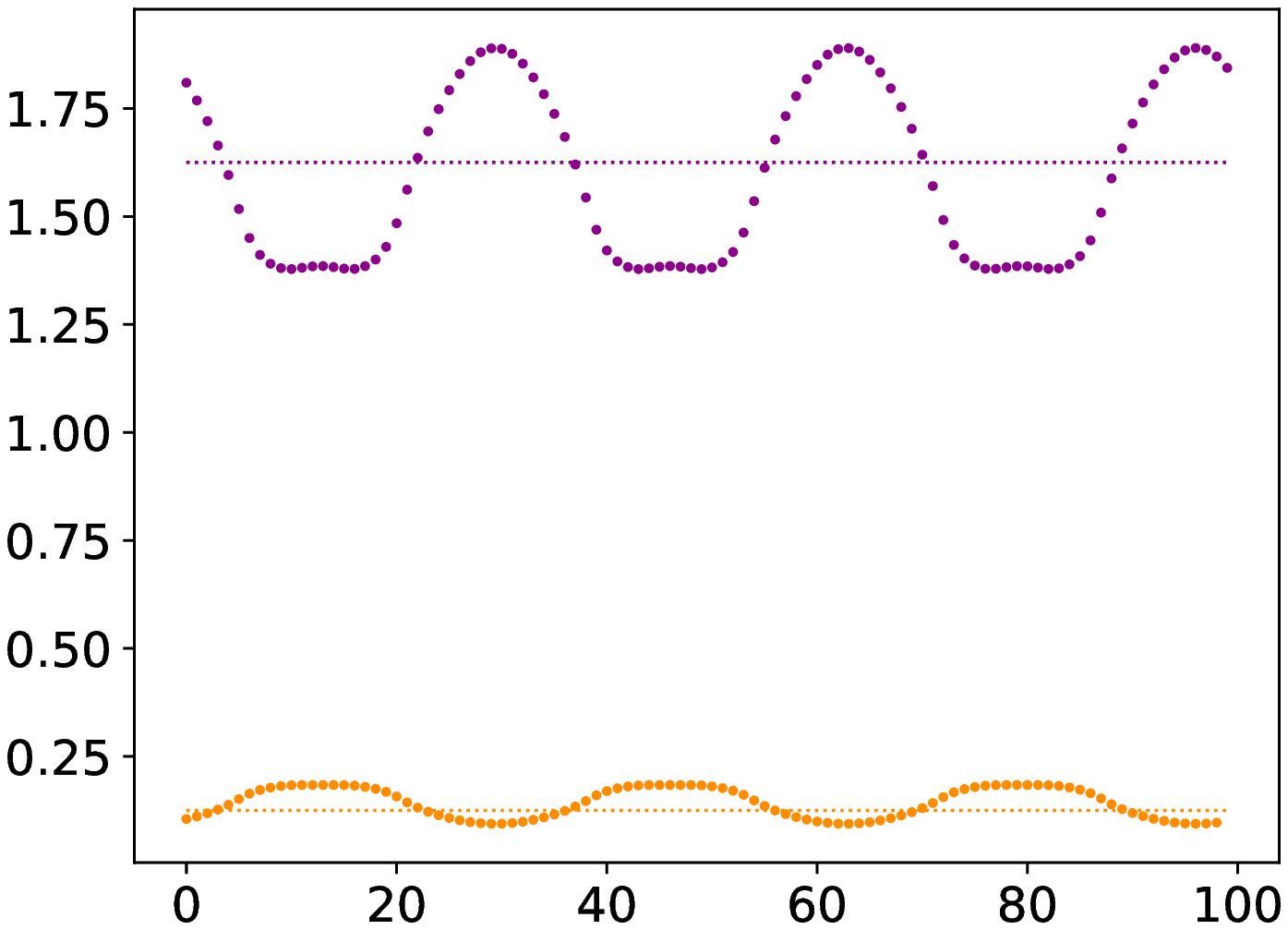}
\put(50,-2){$i$}
\put(0,35){\rotatebox{90}{$u_i,\,v_i$}}			
\put(92,55){$u_*$}	
\put(92,10){$v_*$}
\end{overpic}
}\\
\subfloat[$K=15$\label{RING_N100_K15}]{
\begin{overpic}[width=0.5\textwidth]{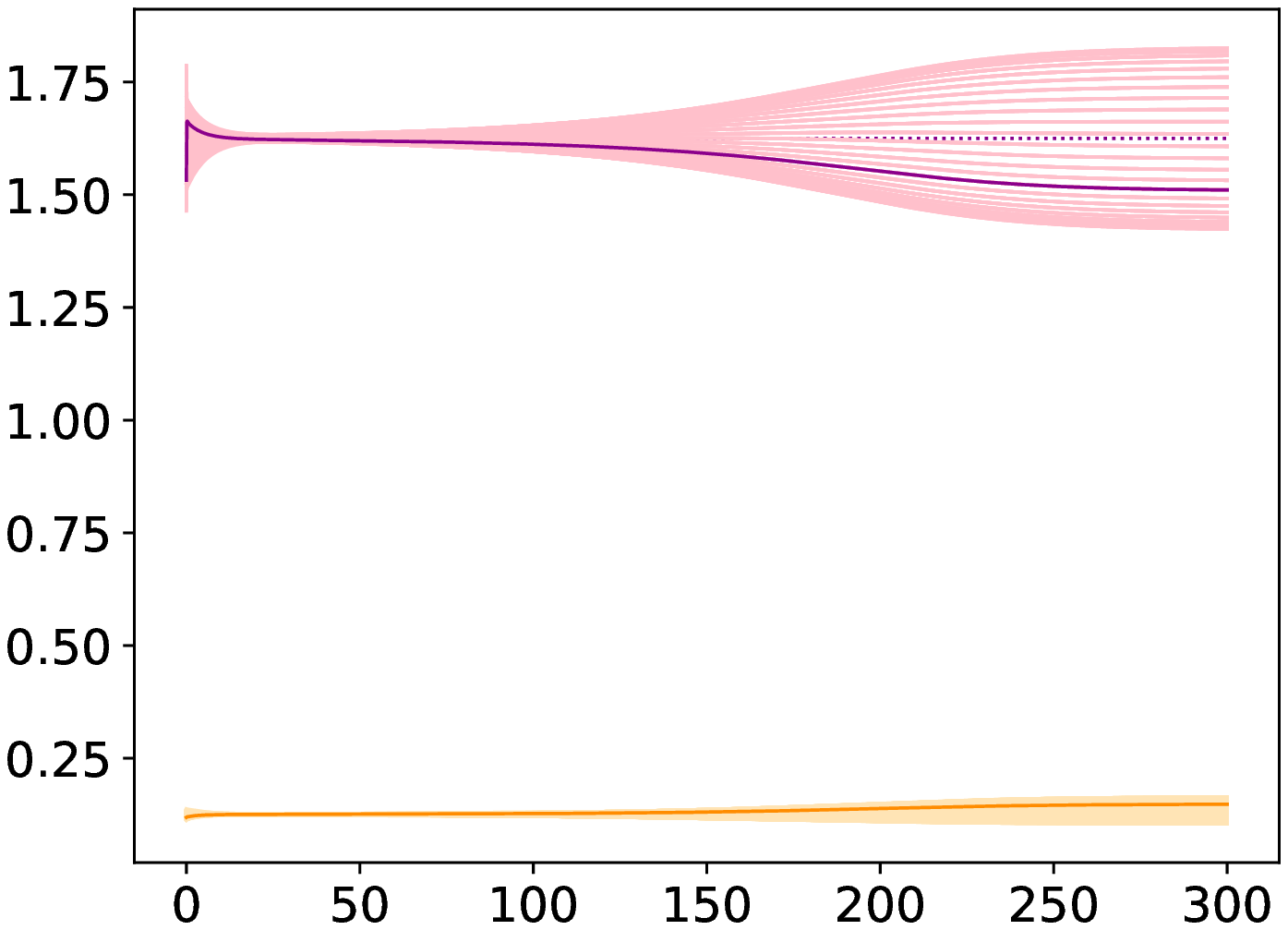}
\put(50,-2){$t$}
\put(0,35){\rotatebox{90}{$u_i,\,v_i$}}		
\put(92,55){$u_*$}	
\put(92,10){$v_*$}	
\end{overpic}
\begin{overpic}[width=0.5\textwidth]{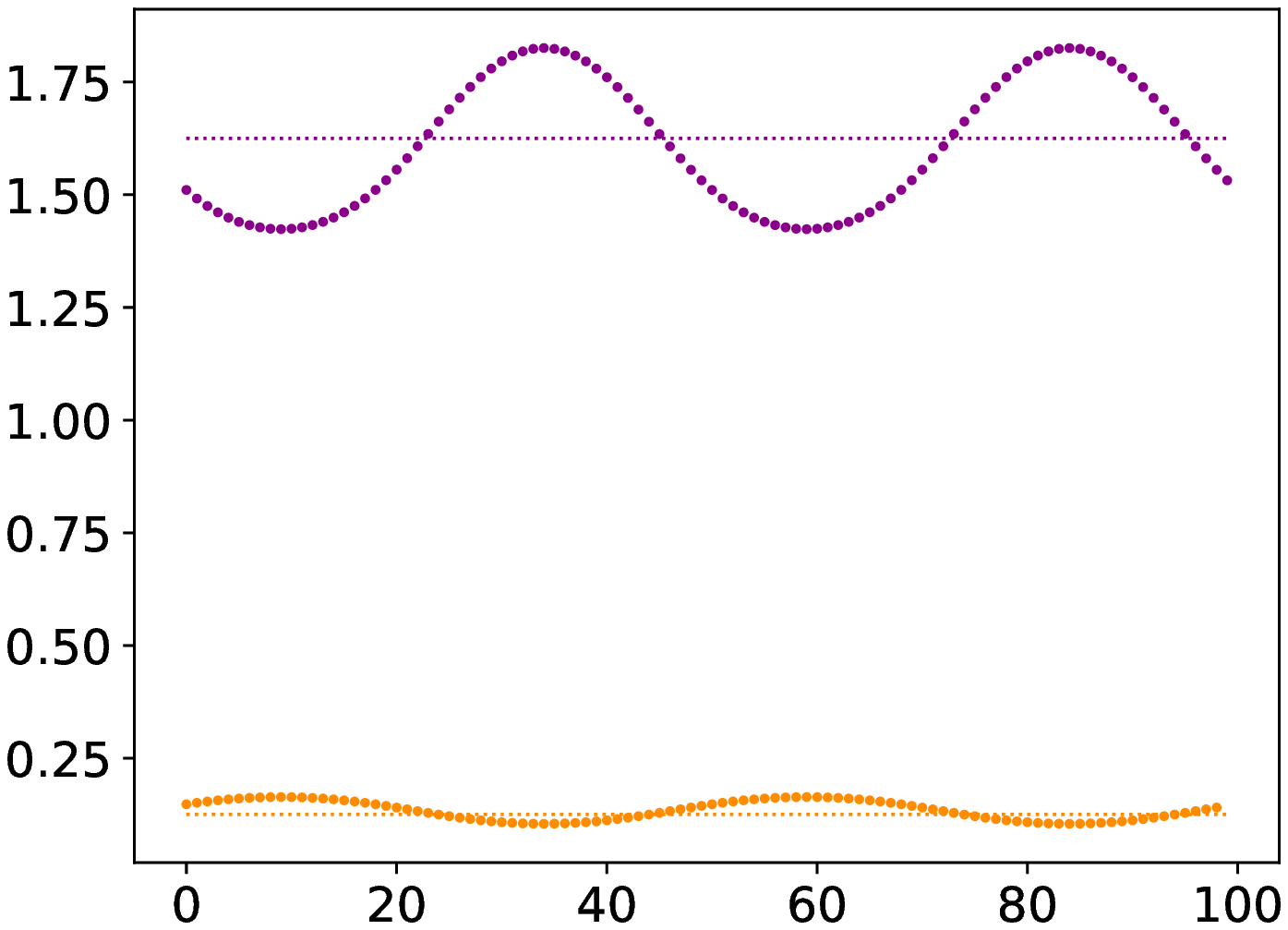}
\put(50,-2){$i$}	
\put(0,35){\rotatebox{90}{$u_i,\,v_i$}}		
\put(92,55){$u_*$}	
\put(92,10){$v_*$}	
\end{overpic}
}\\
\subfloat[$K=20$\label{RING_N100_K20}]{
\begin{overpic}[width=0.5\textwidth]{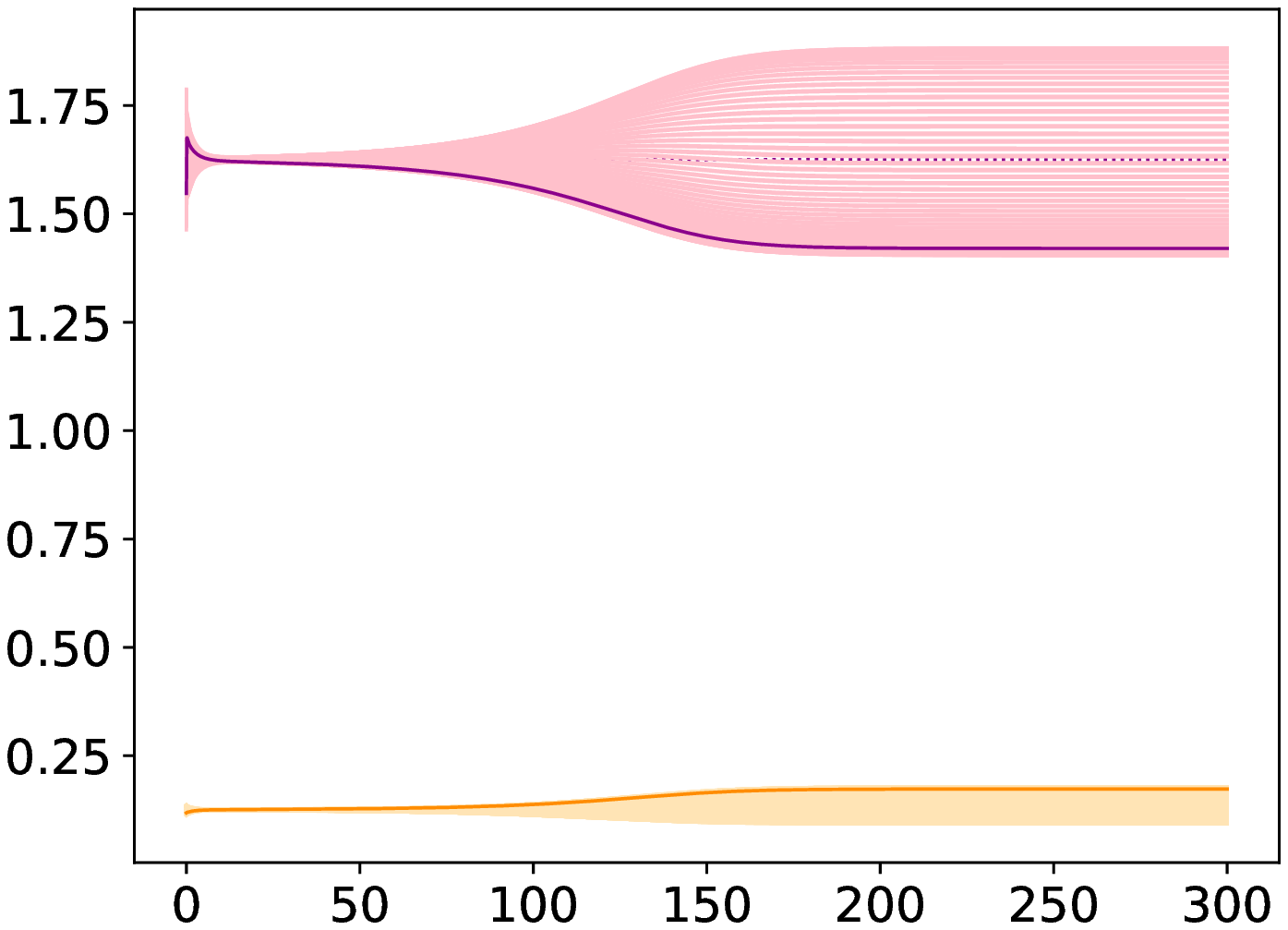}
\put(50,-2){$t$}
\put(0,35){\rotatebox{90}{$u_i,\,v_i$}}		
\put(92,55){$u_*$}	
\put(92,10){$v_*$}
\end{overpic}
\begin{overpic}[width=0.5\textwidth]{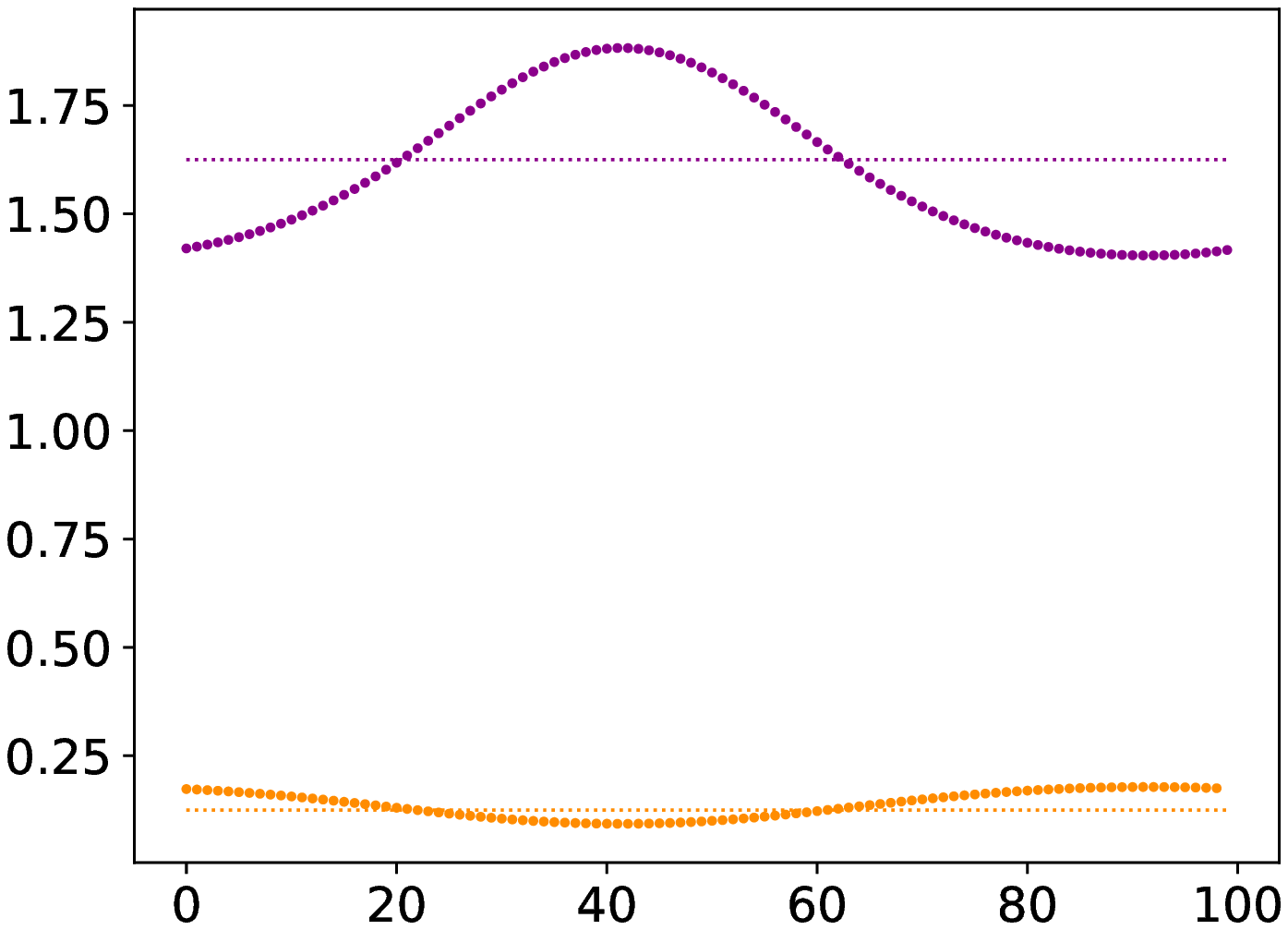}
\put(50,-2){$i$}	
\put(0,35){\rotatebox{90}{$u_i,\,v_i$}}		
\put(92,55){$u_*$}	
\put(92,10){$v_*$}	
\end{overpic}
}
\caption{Dynamics on a $2K$-regular ring of $N=100$ nodes for different values of $K$. The parameter set is reported in Table~\ref{tab:SKTparam} and $d=0.03$. Initial conditions are random perturbations of the homogeneous state. Purple and orange correspond to $u$ and $v$, respectively. \emph{(Left)} Dynamics over time. The trajectory of one node is highlighted, while the others are shown in a lighter color; each node is identified via its index $i\in\{1,2,\ldots, N\}$. \emph{(Right)} Stable configurations on the ring. Different kinds of patterns appear.}\label{fig:ring_dyn_NK}
\end{figure}

\subsection{Cross-diffusion induced instability and graph topology}
This section is devoted to showing, how the cross-diffusion terms are related to the graph topology. We consider different types of graphs, including random graphs such as \textit{small world} and \textit{Erd{\H{o}}s--R{\'e}nyi} networks. The same parameter set describes the dynamics, but different outcomes emerge as the result of the underlying topology. As a key factor, the spectral properties of the graph Laplacian are compared. 

\subsubsection{2D-Lattices}
We consider here three different two-dimensional grid graphs: triangular, square and hexagonal lattices. In a triangular lattice graph, each square unit has a diagonal edge. The square lattice has each node connected to its four nearest neighbors, while in the hexagonal lattice nodes and edges are the hexagonal tiling of the plane. These three structures constitute three different network topologies for the same set of nodes. Note that, the degree is the same for all the nodes (except for ``boundary nodes''): 6 in the triangular lattice, 4 in the square lattice and 3 in the hexagonal lattice. 

In Figure \ref{fig:lattice_346} we show the location of the spectra of the three topologies ($N=110$) with respect to the instability region (obtained using the parameter set in Table \ref{tab:SKTparam}). Note that the value of the largest eigenvalue does not change significantly increasing $N$, since the lattices are almost regular. It can be observed that nonhomogeneous steady states cannot appear on the hexagonal grid, while in the triangular and square lattices, the homogeneous steady state is unstable. The steady patterns on a triangular and square lattice with $N=400$ nodes are shown in Figure \ref{fig:LATTICE_34_steady_u} with respect to the $u$ variable.

\begin{figure}
\begin{center}
\begin{overpic}[width=0.5\textwidth]{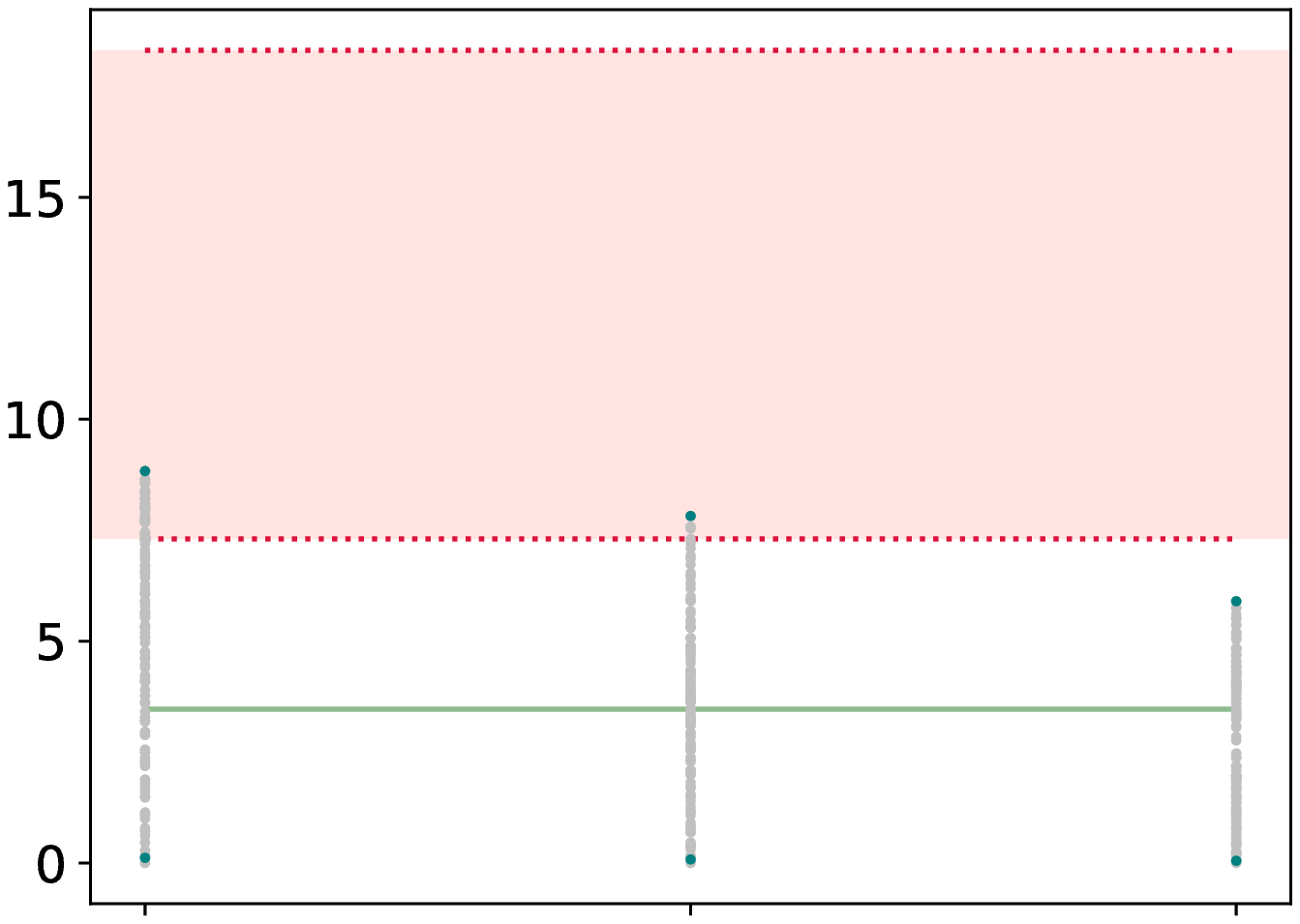}
\put(0,35){$\Lambda$}		
\put(15,2){$\triangle$}	
\put(50,2){$\square$}
\put(85,2){\Large{$\varhexagon$}}
\put(92,30){$\Lambda_{*2}$}	
\put(92,60){$\Lambda_{*1}$}	
\put(92,20){$\Lambda_{*}$}
\end{overpic}
\end{center}
\caption{Cross-diffusion driven instability on triangular ($\triangle$), square ($\square$) and hexagonal ($\varhexagon$) lattices (with $N=110$). Grey dots mark the eigenvalues, while the first and the last nonzero eigenvalues appear in green. The instability region $\Omega_*=(\Lambda_{*1},\Lambda_{*2})$ and the threshold $\Lambda_*$ are marked with the red stripe and the green horizontal solid line, respectively (relevant to the parameter set in Table~\ref{tab:SKTparam} and $d=0.03$).}\label{fig:lattice_346}

\subfloat[triangular\label{LATTICE_3_steady_u}]{
\begin{overpic}[width=0.5\textwidth]{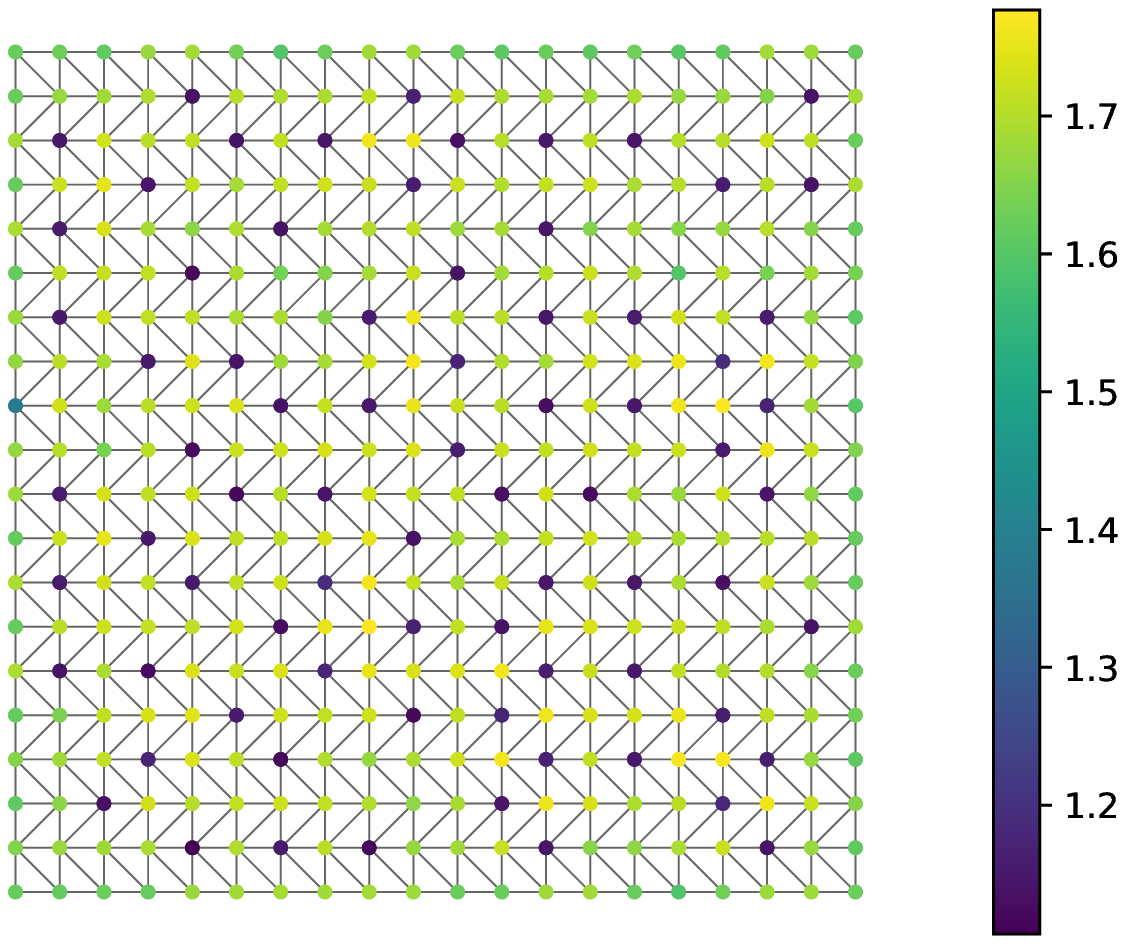}
\end{overpic}
}
\subfloat[square\label{LATTICE_3_steady_u1}]{
\begin{overpic}[width=0.5\textwidth]{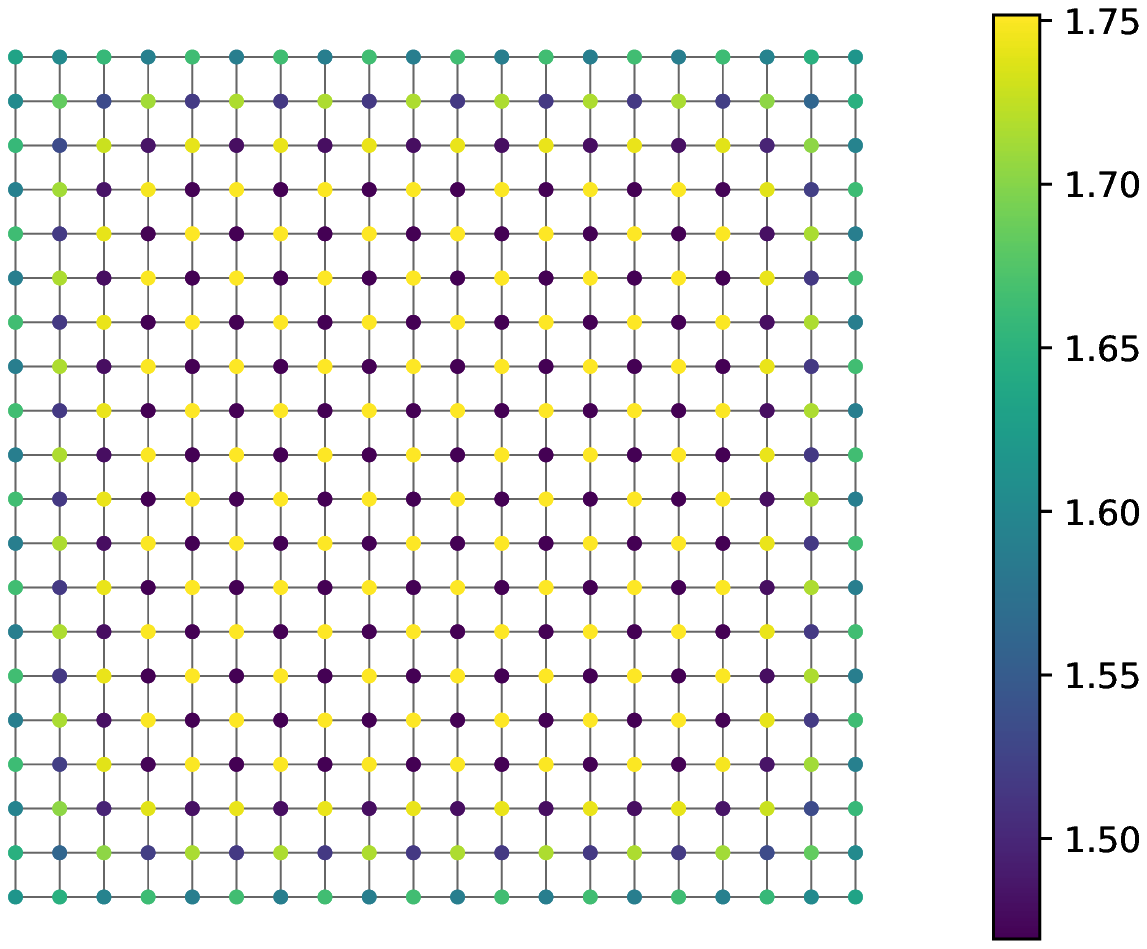}
\end{overpic}
}
\caption{Steady state ($u$ variable) on a triangular and square lattices with $N=400$ nodes. The parameter set is reported in Table~\ref{tab:SKTparam} and $d=0.03$.  Initial conditions are random perturbations of the homogeneous state.}\label{fig:LATTICE_34_steady_u}
\end{figure}
\FloatBarrier
\subsubsection{Random graphs}
We now turn our attention to several random graphs. Also in this case we want to study the influence of a particular structure on the emergence of nonhomogeneous steady states in the network. As in the previous section, we look at the nonzero eigenvalues of the graph Laplacian. Of course, since we are now dealing with random graphs, we consider the distribution of the eigenvalues (mean values and the corresponding variance) obtained with several realizations of the same random graph. 

\medskip
\noindent
We consider the following random graphs (briefly recalling the definition and some properties).
\begin{itemize}
\item[-] \emph{Regular-random graph}\\
A random-regular graph is chosen uniformly from the set of all $K$-regular graphs with $N$ nodes.
\item[-] \emph{small-world graph (Watts--Strogatz)} \cite{watts1998collective} 
Starting from a ring lattice with $N$ nodes and $K$ edges per node, each edge is rewired at random with probability $p$. This construction allows to ``tune'' the graph between $K$-regular rings (p=0) and random graphs ($p=1$). The small-world region ($10^{-4}<p<10^{-1}$) is characterized by a small average path length and a large clustering coefficient.
\item[-] \emph{binomial graph (Erd\H{o}s--R\'eyni}) \cite{erdos1960evolution}
Considering $N$ nodes, each of the possible edges is chosen with probability $p$. 
\item[-] \emph{preferential attachment model (Barab\'asi--Albert)} \cite{barabasi1999emergence}
A graph of $N$ nodes is grown by attaching new nodes each with $K$ edges that are preferentially attached to existing nodes with a high degree.
\end{itemize}

We compare these different graph topologies and the appearance of cross-diffusion-induced instability. As in the previous sections, we are interested in the spectrum of the graph Laplacian. Since we deal with particular classes of random graphs, general results cannot be achieved by just looking at a particular realization. For each type of structure, we fix the number of nodes $N=100$ varying the other network parameters (the probability $p$ for Watts-Strogatz and Erd\H{o}s--R\'eyni graphs, the number of edges of a node for random-regular and Barab\'asi--Albert graphs). We generate 1000 realizations for each structure and each parameter value, obtaining a distribution of the eigenvalues. In Figure \ref{fig:graphtopology_spectrum}, the mean (dots) and the variance (error bars) of the eigenvalues of the graph Laplacian are shown for the four types of structures. The first nontrivial eigenvalue and the last one are marked in green. The red stripe denotes the instability region $(\Lambda_{*1},\Lambda_{*2})$, while the green solid horizontal line denote the threshold $\Lambda_{*}$, relevant to the parameter set in Table~\ref{tab:SKTparam}. 

In a regular-random graph (Figure \ref{RR_network}) it can be seen that the cross-diffusion induced instability appears on average in the region when $2\leq K/2 \leq 13$ and $N=100$, while greater values of the node-degree only lead to a very small probability of instability and the uniform steady state is expected to be stable. In a small-world network, with $K/2=15$ being fixed, we report in Figure \ref{SW_network} the averages and variances of the  eigenvalues for different values of the probability $p$ of rewiring an edge in a larger interval than the small-world regime. It turns out that the systems show with a probability well bounded away from zero non-uniform steady states for the considered values of $p$ because at least one averaged eigenvalue is always located well within the instability region. Depending on the number of eigenvalues in the instability region, different types of non-uniform steady states can be observed. For the Erd{\H{o}}s--R{\'e}nyi graph, cross-diffusion induced instability is most likely to appear for small values of the probability $p$ of choosing an edge, while it is not likely to occur for larger values. In order to better compare this structure with the other types, we consider $p\approx4K/(N-1)$, which for large $N$ approximates the average number of links for one node. With $N=100$ and $K=30$, we obtain $p\approx 0.3$, for which the systems display cross-diffusion induced instability with a probability well bounded away from zero. Finally, the Barab\'asi--Albert graph is quite different from the others. The spectrum has a large variance, and especially for large values of $K$ (that in this case represents the number of edges chosen for every new node in the growing process). We observe that the possibility of pattern formation really depends much more on the particular realization of the graph in comparison to the other graph structures.

A final observation has to be made on the non-uniform steady state that arises through cross-diffusion induced instability. Its shape is still quite regular and ``smooth'' as reported for regular rings in the previous section, but small variations from this regular configuration appear (see Figure \ref{fig:SW_steady_p}). For the other structures instead, the system tends to a stable non-uniform configuration that however does not present the characteristic shape of valleys and bumps.

\medskip
In addition to the phenomenological discussion, ecological back-interpretation is also important. The first crucial observation is that cross-diffusion enables the coexistence of the species with a non-homogeneous distribution. In fact, without cross-diffusion, the homogeneous distribution of the populations (namely all nodes reach the same steady state) is the only outcome. This fact is indeed important in ecology: homogeneity and synchronization of metacommunities may be harmful to species' survival. Cross-diffusion also affects the populations' abundance as observed in the PDEs model. Looking at the total abundance of the species on the network for the simulations presented in the papers and the parameter set considered, we observe a small decrease in the total abundance of population $u$ and an increase in the total abundance of population $v$ with respect to the homogenous case (without cross-diffusion). For instance, for the ring topology ($N=100,\, K=10$), parameter set as in Table \ref{tab:SKTparam} and $d=0.03$, species $u$ decreases of 1.78\% while species $v$ increases of 12.3\%. This quantitative difference reflects the difference in the order of magnitude of the population sizes at the homogeneous steady state (see for instance Figure \ref{fig:ring_dyn_NK}) and it is a limitation of the particular parameter set. However, the qualitative trend highlights a counter-intuitive effect driven by cross-diffusion, namely that although species $u$ tries to avoid $v$, it is actually $v$ that benefits from this. However, it is worthwhile to say that the total population abundance is not the only factor measuring the advantages/disadvantages of the two competing populations.

\begin{figure}
\begin{center}
\subfloat[regular-random\label{RR_network}]{
\begin{overpic}[width=0.2\textwidth, trim=1cm 0cm 0.9cm 0cm, clip]{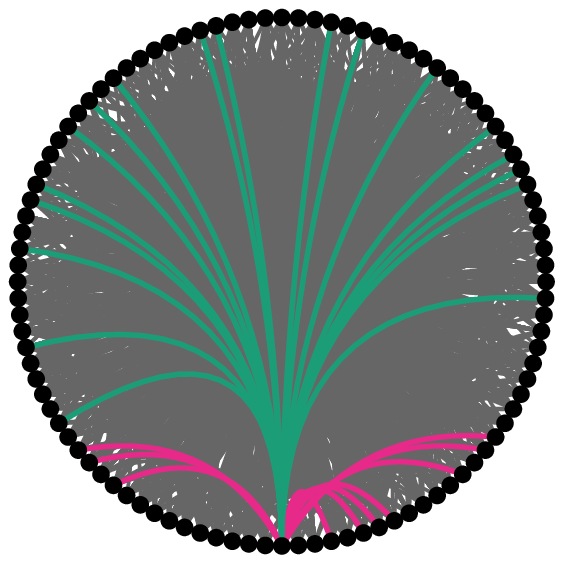}
\end{overpic}
}
\subfloat[small-world (K/2=15)\label{SW_network}]{
\begin{overpic}[width=0.2\textwidth, trim=1cm 0cm 0.9cm 0cm, clip]{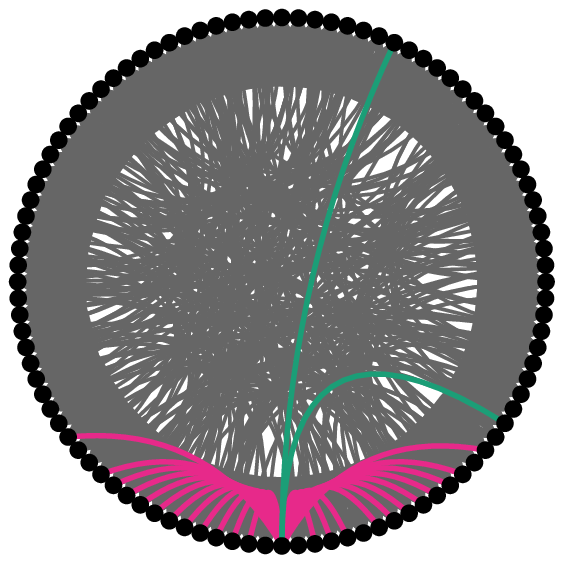}
\end{overpic}
}
\subfloat[Erd{\H{o}}s--R{\'e}nyi \label{ER_network}]{
\begin{overpic}[width=0.2\textwidth, trim=1cm 0cm 0.9cm 0cm, clip]{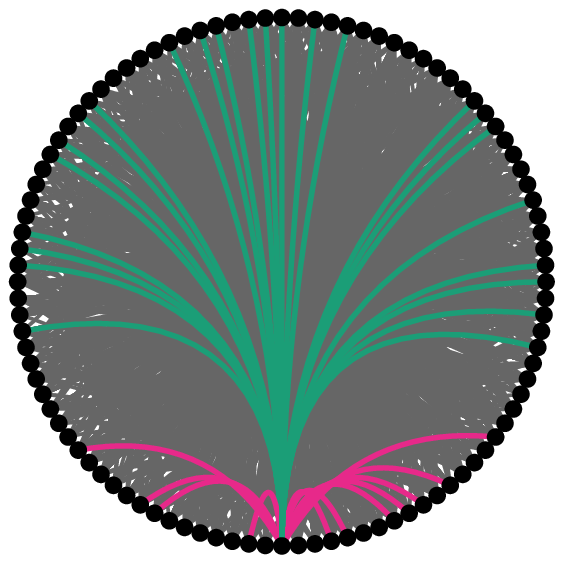}
\end{overpic}
}
\subfloat[Barabasi--Albert\label{BA_network}]{
\begin{overpic}[width=0.2\textwidth, trim=1cm 0cm 0.9cm 0cm, clip]{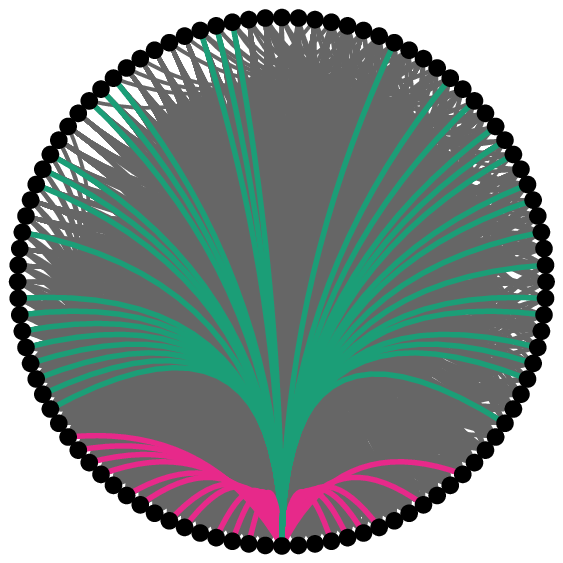}
\end{overpic}
}
\end{center}
\caption{Comparison of different graph topologies (regular-random, small-world,  Erd{\H{o}}s--R{\'e}nyi and Barabasi--Albert). A particular realization of the random graph ($N=100$) is reported, in which the edges between ``neighbors node'' ($K/2=15$ on each side) appear in magenta, while ``long-range interaction'' in green.}\label{fig:graphtopology_structure}
\end{figure}

\begin{figure}
\begin{center}
\subfloat[regular-random\label{RR_spectrum}]{
\begin{overpic}[width=0.5\textwidth]{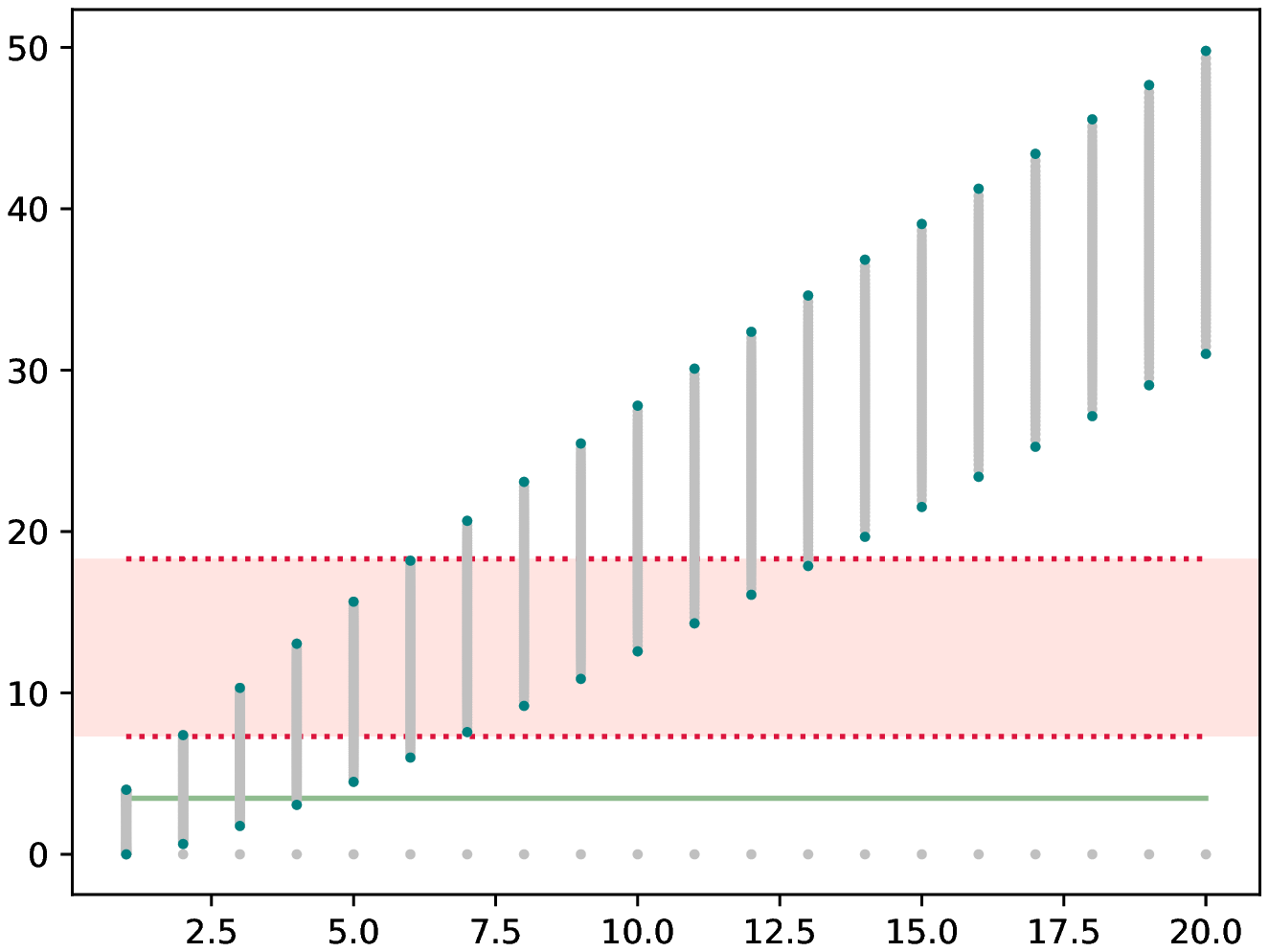}
\put(0,35){$\Lambda$}		
\put(83,0){$K/2$}	
\put(92,30){$\Lambda_{*2}$}	
\put(92,12){$\Lambda_{*}$}	
\put(92,18){$\Lambda_{*1}$}
\end{overpic}
}
\subfloat[small-world\label{SW_spectrum}]{
\begin{overpic}[width=0.5\textwidth]{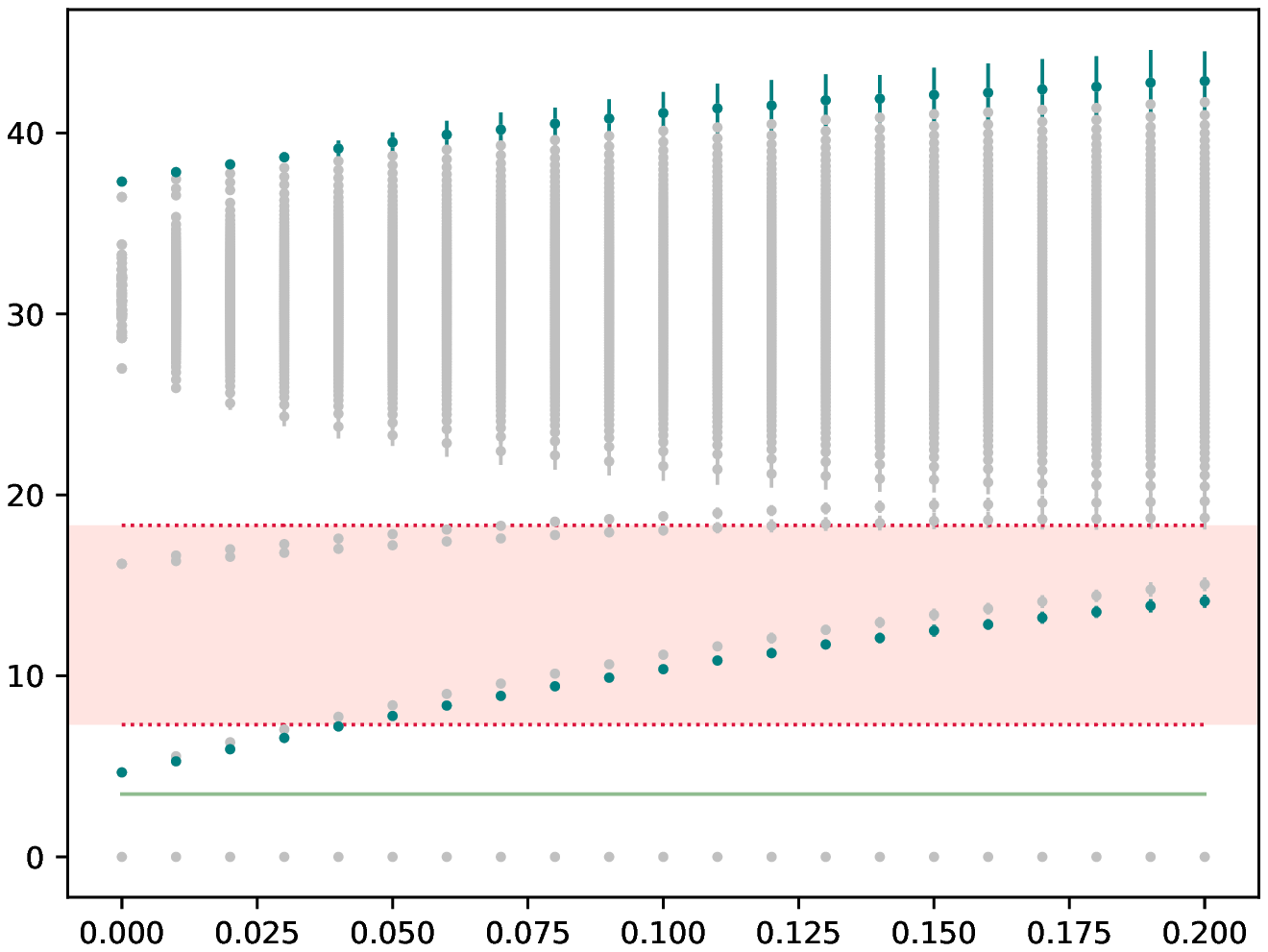}
\put(0,35){$\Lambda$}		
\put(83,0){$p$}	
\put(92,30){$\Lambda_{*2}$}	
\put(92,12){$\Lambda_{*}$}	
\put(92,18){$\Lambda_{*1}$}
\end{overpic}
}\\
\subfloat[Erd{\H{o}}s--R{\'e}nyi\label{ER_spectrum}]{
\begin{overpic}[width=0.5\textwidth]{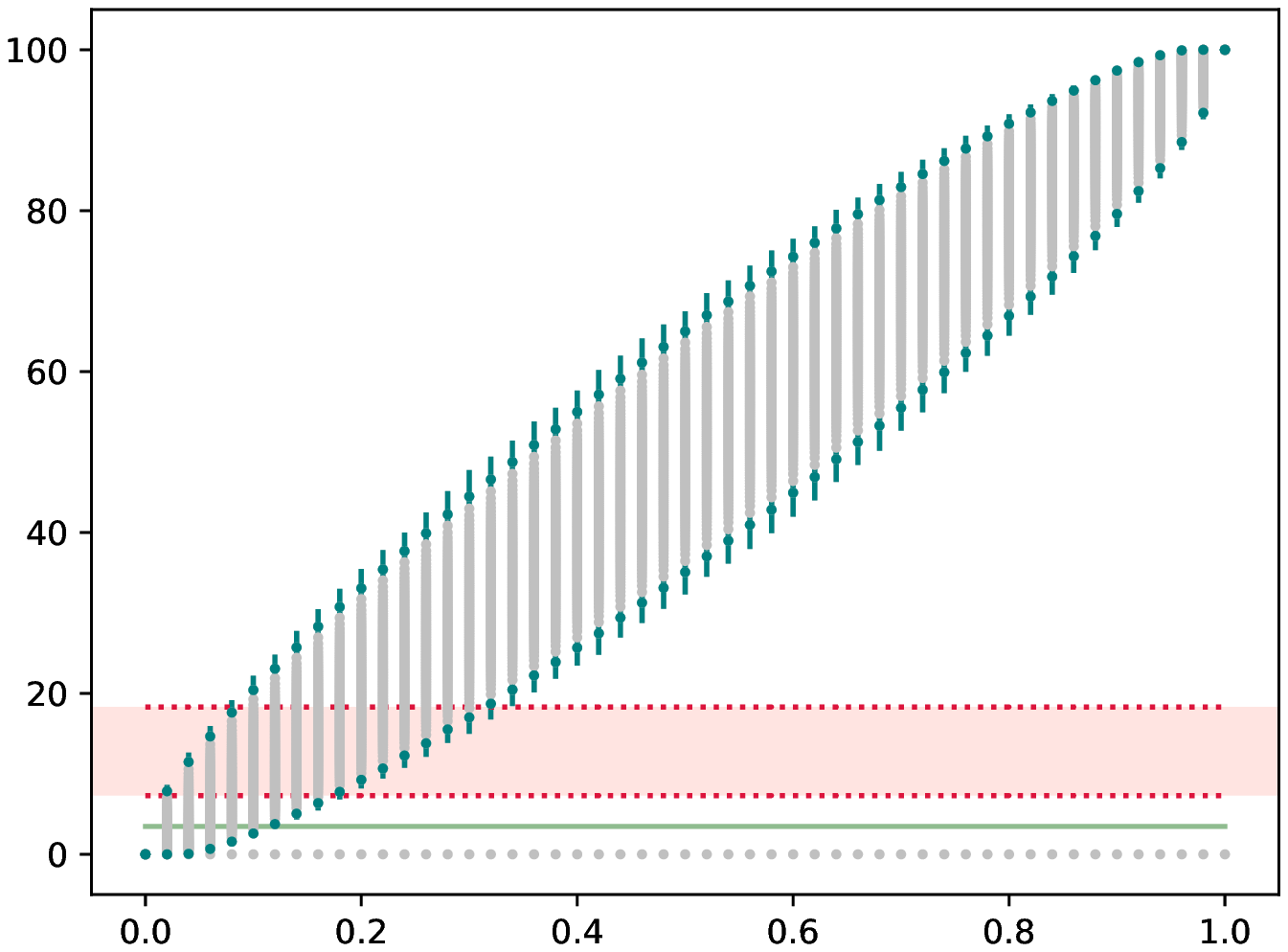}
\put(0,35){$\Lambda$}		
\put(83,0){$p$}	
\put(92,20){$\Lambda_{*2}$}	
\put(92,10){$\Lambda_{*}$}	
\put(92,15){$\Lambda_{*1}$}
\end{overpic}
}
\subfloat[Barabasi--Albert\label{BA_spectrum}]{
\begin{overpic}[width=0.5\textwidth]{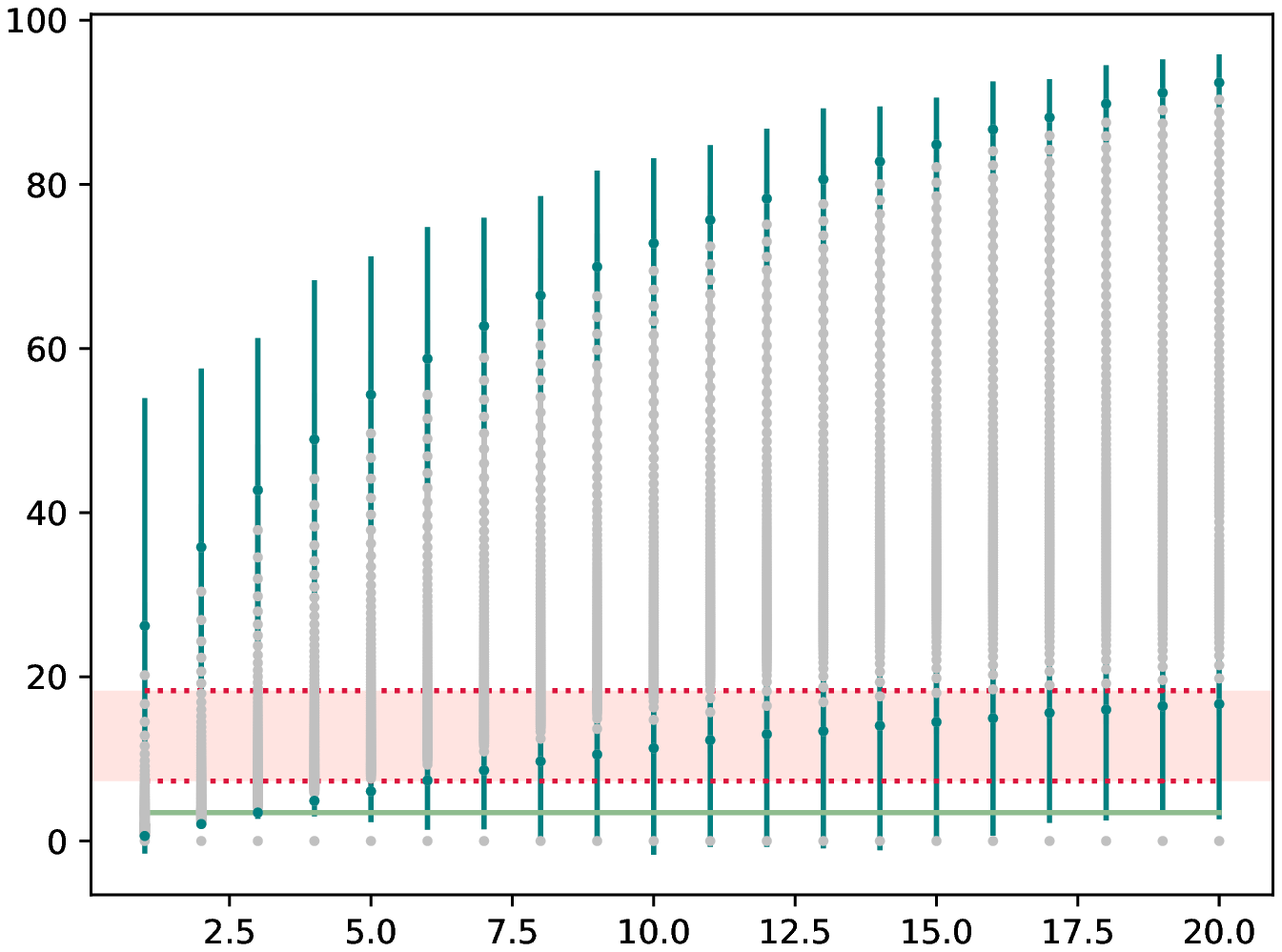}
\put(0,35){$\Lambda$}		
\put(83,0){$K/2$}	
\put(92,20){$\Lambda_{*2}$}	
\put(92,10){$\Lambda_{*}$}	
\put(92,15){$\Lambda_{*1}$}
\end{overpic}
}\\
\end{center}
\caption{Comparison of different graph topologies and the appearance of cross-diffusion induced instability. The mean and variance of the eigenvalues of the graph Laplacian, obtained with 1000 realization of the random graphs are denoted by dots and error bars, respectively; the first nontrivial eigenvalue and the last one are marked in green. The red stripe denotes the instability region $(\Lambda_{*1},\Lambda_{*2})$, while the green solid horizontal line denotes the threshold $\Lambda_{*}$, relevant to the parameter set in Table~\ref{tab:SKTparam} and $d=0.03$.}\label{fig:graphtopology_spectrum}
\end{figure}

\begin{figure}
\subfloat[$p=0.01$\label{SW_d03_p01_steady}]{
\begin{overpic}[width=0.5\textwidth]{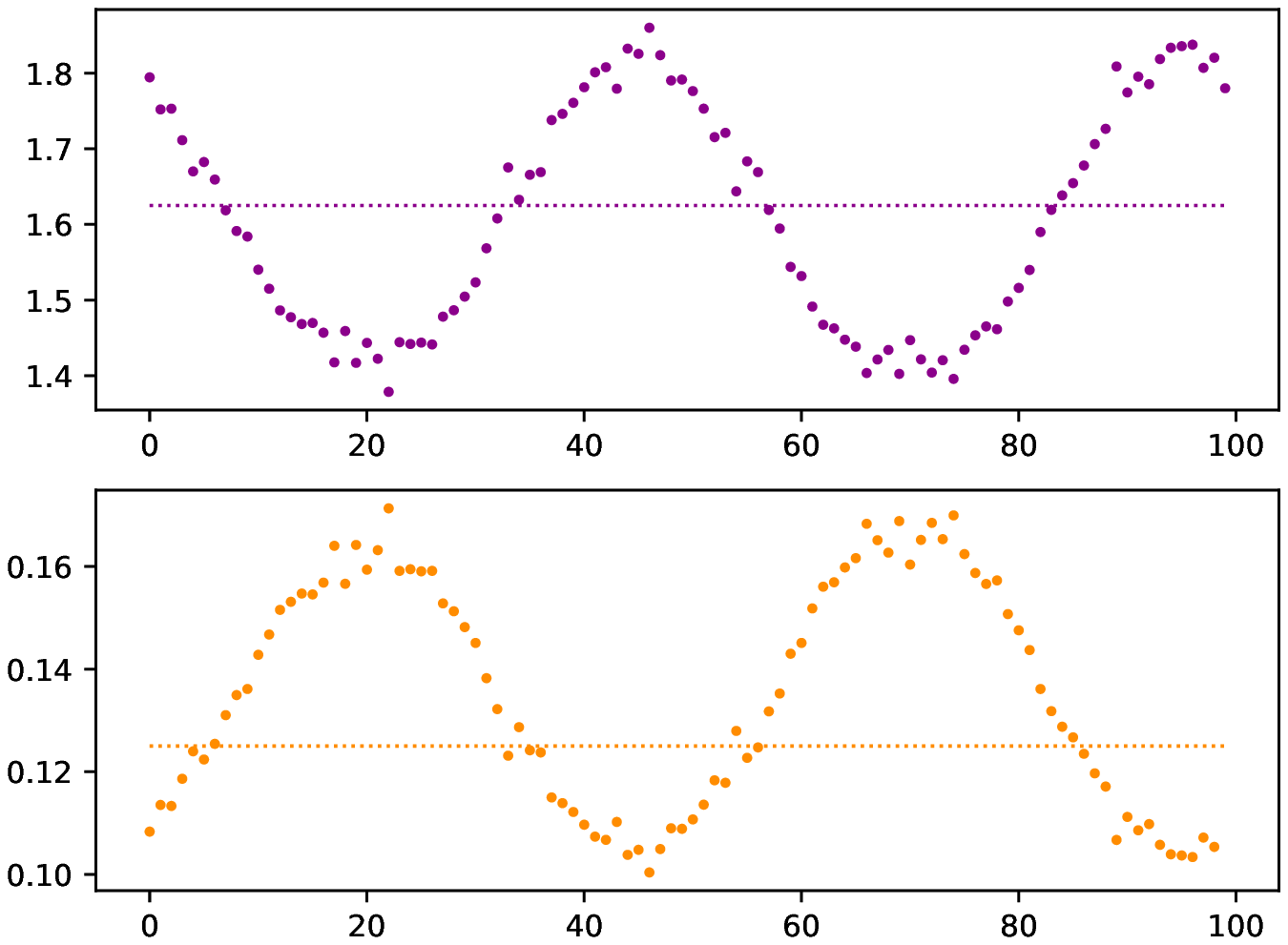}
\put(90,0){$i$}	
\put(0,20){\rotatebox{90}{$v_i$}}
\put(0,50){\rotatebox{90}{$u_i$}}				
\put(92,52){$u_*$}	
\put(92,17){$v_*$}	
\end{overpic}
}
\subfloat[$p=0.05$\label{SW_d03_p05_steady}]{
\begin{overpic}[width=0.5\textwidth]{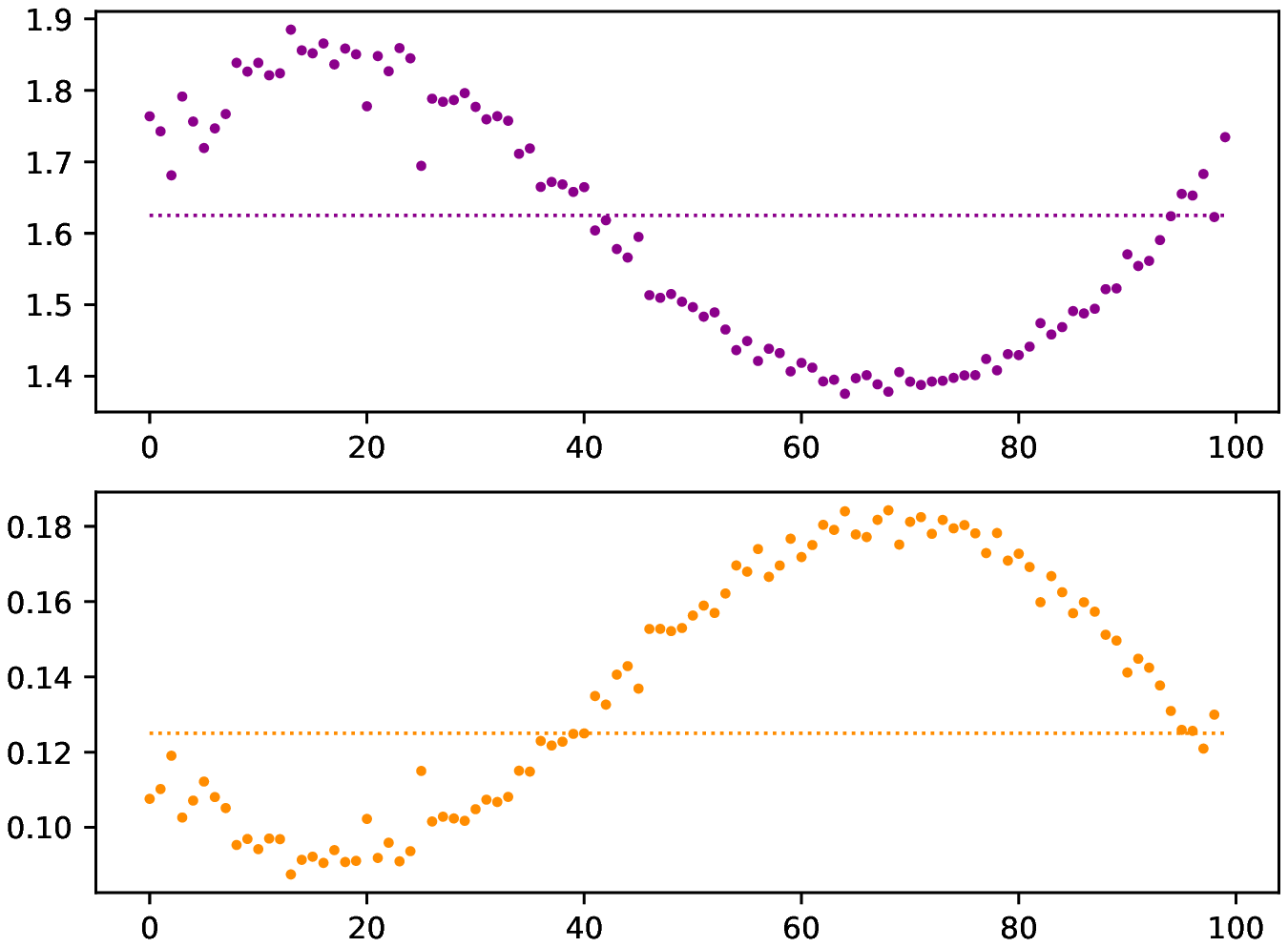}
\put(90,0){$i$}	
\put(0,20){\rotatebox{90}{$v_i$}}
\put(0,50){\rotatebox{90}{$u_i$}}				
\put(92,52){$u_*$}	
\put(92,17){$v_*$}	
\end{overpic}
}
\caption{Non-homogeneous steady state on a small-world network with $N=100$ and $K=30$ for different values of the probability $p$. The parameter set is reported in Table~\ref{tab:SKTparam} and $d=0.03$. Initial conditions are random perturbations of the homogeneous state.}\label{fig:SW_steady_p}
\end{figure}

\section{Conclusion and Outlook} \label{sec:conclusion}
In this paper, we have extended the cross-diffusion induced instability, already known in reaction--cross-diffusion systems of PDEs, to networks of dynamical systems. The non-standard diffusion terms are also written using the graph Laplacian. We established the general framework and we have shown that in this case the linearization slightly differs from the standard case and the PDEs case. Then, we adapted and investigated the SKT cross-diffusion model for competing species on a network. As already known in the context of reaction--cross-diffusion models, the cross-diffusion terms can destabilize the homogeneous equilibrium state and cause the appearance of organized states and patterns, which is not possible only with standard diffusion terms. The important finding is that the obtained conditions for cross-diffusion induced instability in the network framework are the same as the continuous case, where the role of the eigenvalues of the Laplace operator is played by the eigenvalues of the graph Laplacian. Also in this case, an instability region depending on the model parameter characterizes the possibility to observe patterns, and the location of the eigenvalues of the graph Laplacian determines their presence and shape. Finally, we have analyzed different network structures (such as regular rings, 2D lattices and different random graphs) in order to show how the network structure influences the possible outcomes of the system. In particular, we have looked at the spectrum of the graph Laplacian (or its distribution for random graphs). We conclude that cross-diffusion induced instability depends on the network structure: for instance, for the SKT network model in the weak competition case it is more likely to appear on triangular lattices, and on small-world and Barabasi--Albert random graphs.

One key aim of this work is to point out that cross-diffusion terms can be a useful tool to model complex systems, and that they can give rise to richer dynamics. Several research directions arise at this point. On the one hand, we can look at the dynamical and bifurcation aspects of this topic. As widely investigated for cross-diffusion PDEs systems, we want to find an entropy functional or a Lyapunov function for networks in order to achieve global stability results \cite{belykh2004connection, li2010global, slavik2021reaction}. Furthermore, a deeper investigation of the system outcomes combined with the bifurcation structure (that can be computed by the continuation software \texttt{pde2path} \cite{uecker2021pde2path}) may reveal the presence of periodic patterns, as in the PDEs case. On the other hand, one could take an even more detailed look at the influence of the graph topology. In this regard, the parallel between continuous and discrete models is intriguing. Inspired by the discretized version of the PDEs model involving the Laplace operator, we considered in this paper different graph topologies, changing completely the \textit{corresponding} operator in the continuous case. Therefore, it would be interesting to understand the limit model back to the continuous case. Moreover, following \cite{asllani2014theory}, it seems possible to extend the analysis to directed networks. Moreover, considering for instance the SKT model on networks, it can also be possible to consider the dynamics of the two competing species evolving on different networks, as sketched in Figure \ref{fig:multilayer}. This would lead us to consider the dynamics on a multilayer network~\cite{aleta2019multilayer, brechtel2018master, kivela2014multilayer}, where the two layers can be different, for which the extension of the theory of Turing patterns has already been presented in~\cite{asllani2014turing, brechtel2018master,kouvaris2015pattern}. It could also be interesting to study other types of cross-diffusion terms involving different nodes: for instance, in the SKT model, the movement of one species from node $i$ to node $j$ is influenced by the presence of the same species or the competing one on node $j$. In this scenario, there is an information flow or knowledge about the status of the other nodes, or it can be seen as a weighted network with link weights depending on the quantities on each node. Finally, one could study possible applications of cross-diffusion systems on networks, ranging from ecology and landscape modelling to disease spreading~\cite{duan2019turing, lang2018analytic}, that are characterized by a strong interplay between dynamics and network structure.

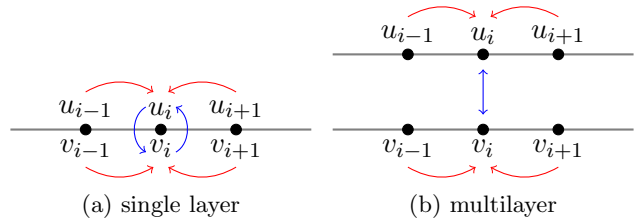
\begin{figure}
\begin{center}
\subfloat[single layer\label{1layernetwork}]{
\begin{tikzpicture}
\draw[gray, thick] (-2,0) -- (2,0);
\filldraw[black] (-1,0) circle (2pt);
\filldraw[black] (0,0) circle (2pt); 
\filldraw[black] (1,0) circle (2pt);
\draw[anchor=north](0,0) node {$v_i$};
\draw[anchor=south](0,0) node {$u_i$};
\draw[anchor=north](-1,0) node {$v_{i-1}$};
\draw[anchor=south](-1,0) node {$u_{i-1}$};
\draw[anchor=north](1,0) node {$v_{i+1}$};
\draw[anchor=south](1,0) node {$u_{i+1}$};

\draw [->,red] (-1,0.5) to [out=30,in=150] (-0.1,0.5);
\draw [->,red] (1,0.5) to [out=150,in=30] (0.1,0.5);

\draw [->,red] (-1,-0.5) to [out=330,in=210] (-0.1,-0.5);
\draw [->,red,bend right] (1,-0.5) to [out=30,in=150] (0.1,-0.5);

\draw [->,blue] (-0.2,0.3) to [out=210,in=150] (-0.2,-0.3);
\draw [->,blue] (0.2,-0.3) to [out=30,in=330] (0.2,0.3);
\end{tikzpicture}
}
\subfloat[multilayer\label{multilayer}]{
\begin{tikzpicture}
\draw[gray, thick] (-2,0) -- (2,0);
\filldraw[black] (-1,0) circle (2pt);
\filldraw[black] (0,0) circle (2pt); 
\filldraw[black] (1,0) circle (2pt);

\draw[gray, thick] (-2,1) -- (2,1);
\filldraw[black] (-1,1) circle (2pt);
\filldraw[black] (0,1) circle (2pt); 
\filldraw[black] (1,1) circle (2pt);

\draw[anchor=north](0,0) node {$v_i$};
\draw[anchor=south](0,1) node {$u_i$};
\draw[anchor=north](-1,0) node {$v_{i-1}$};
\draw[anchor=south](-1,1) node {$u_{i-1}$};
\draw[anchor=north](1,0) node {$v_{i+1}$};
\draw[anchor=south](1,1) node {$u_{i+1}$};

\draw [->,red] (-1,1.5) to [out=30,in=150] (-0.1,1.5);
\draw [->,red] (1,1.5) to [out=150,in=30] (0.1,1.5);

\draw [->,red] (-1,-0.5) to [out=330,in=210] (-0.1,-0.5);
\draw [->,red,bend right] (1,-0.5) to [out=30,in=150] (0.1,-0.5);

\draw [<->,blue] (0,0.2) -- (0,0.8);
\end{tikzpicture}
}
\end{center}
\caption{Different way to model the dynamics of two species on a network: the species evolve, compete and move \protect\subref{1layernetwork} on the same network, or \protect\subref{multilayer} on a multilayer network.\label{fig:multilayer}}
\end{figure}
\FloatBarrier

\medskip
\textbf{Acknowledgements:} The authors thank Esther Daus for the fruitful discussion about the topic of the paper. CK has been supported by a Lichtenberg Professorship of the VolkswagenStiftung. CK also acknowledges partial support of the EU within the TiPES project funded by the European Unions Horizon 2020 research and innovation programme under Grant Agreement No. 820970. CS has received funding from the European Union's Horizon 2020 research and innovation programme under the Marie Sk\l odowska--Curie Grant Agreement No. 754462. CS also acknowledge the support of IGGSE within the SEND project under Project Teams (14th Cohort) and of the University of Graz through the NAWI Graz Visiting Award 2022 IMSC Young Investigator Fellowship. Support by INdAM-GNFM is gratefully acknowledged by CS. 
\bibliographystyle{plainurl}
\bibliography{bibliography}

\appendix
\section{Discretized SKT reaction--diffusion model}\label{A:discrSKT}
The extension of cross-diffusion systems on networks comes from the discretization of reaction--cross-diffusion systems by finite differences. In particular, we consider the SKT model with self- and cross-diffusion terms given by
\begin{equation}
\begin{cases}
\partial_t u=\Delta((d_1+d_{11} u +d_{12} v)u)+(r_1-a_1 u-b_1 v)u,& \textnormal{on } \mathbb{R}_+\times \Omega,\\[0.1cm]
\partial_t v=\Delta((d_2+d_{22} v +d_{21} u)v)+(r_2-b_2 u-a_2 v)v,& \textnormal{on }  \mathbb{R}_+\times \Omega,\\[0.1cm]
\dfrac{\partial u}{\partial n}=\dfrac{\partial v}{\partial n}=0,&\textnormal{on } \mathbb{R}_+\times \partial\Omega,\\[0.1cm]
u(0,x)=u_{in}(x),\; v(0,x)=v_{in}(x),& \textnormal{on } \Omega,
\end{cases}\label{cross}
\end{equation}
where the quantities $u(t,x),\; v(t,x)\geq 0$ represent the population densities of two species at time $t$ and position $x$, confined and competing for resources on a bounded and connected domain $\Omega~\subset~\mathbb{R}^n$. The parameters are the same as the network model presented in Section \ref{netSKT}.

The discretized SKT model on a 1D domain of length $\ell$ using $N$ nodes and uniform mesh size $h=\ell/(N-1)$ is~\cite{daus2019entropic}
\begin{equation}
\begin{cases}
\dot{u}_i=f(u_i,v_i)
+\dfrac{D_1}{h^2}(u_{i-1}-2u_i+u_{i+1})
+\dfrac{D_{11}}{h^2}(u_{i-1}^2-2u_i^2+u_{i+1}^2)
+\dfrac{D_{12}}{h^2}(u_{i-1}v_{i-1}-2u_iv_i+u_{i+1}v_{i+1}),\\[0.3cm]
\dot{v}_i=g(u_i,v_i)
+\dfrac{D_2}{h^2}(v_{i-1}-2v_i+v_{i+1})
+\dfrac{D_{22}}{h^2}(v_{i-1}^2-2v_i^2+v_{i+1}^2)
+\dfrac{D_{21}}{h^2}(u_{i-1}v_{i-1}-2u_iv_i+u_{i+1}v_{i+1}),\\[0.3cm]
i=1,\dots,N.
\end{cases}
\end{equation}
The uniform mesh of $N$ nodes can be seen as a 1D lattice of $N$ nodes, each one connected to its two nearest neighbors (where the boundary conditions determine the edges of the boundary nodes). Indicating with $l_{ij}, \,i,j=1,\dots,N$ are the elements of the graph Laplacian, the discretized system can be written in the form
\begin{equation}
\begin{cases}
\dot{u}_i=f(u_i,v_i)
-\dfrac{D_1}{h^2}\displaystyle{ \sum_{j=1}^N l_{ij}u_j
-\dfrac{D_{11}}{h^2}\sum_{j=1}^N l_{ij}u_j^2
-\dfrac{D_{12}}{h^2}\sum_{j=1}^N l_{ij}v_ju_j},\\[0.3cm]
\dot{v}_i=g(u_i,v_i)
-\dfrac{D_2}{h^2} \displaystyle{\sum_{j=1}^N l_{ij}v_j
-\dfrac{D_{22}}{h^2}\sum_{j=1}^N l_{ij}v_j^2
-\dfrac{D_{21}}{h^2}\sum_{j=1}^N l_{ij}u_jv_j},\\[0.3cm]
i=1,\dots,N.
\end{cases}
\end{equation}
Note that the graph Laplacian $L\in\mathbb{R}^{N\times N}$ is the usual tridiagonal matrix 
\begin{equation}
L=
\begin{pmatrix}
1 &  -1       &           &          &    \\
-1 & 2       & -1        &          &      \\
    & \ddots & \ddots &\ddots &     \\
    &           & -1        & 2      & -1 \\
    &           &           & -1       & 1\\
\end{pmatrix}
\end{equation}
with eigenvalues given by the closed formula~\cite{yueh2005eigenvalues}
$$\Lambda_j=2-2\cos{\left(\dfrac{\pi(j-1)}{N}\right)},\qquad j=1,\dots,N.$$
The main difference with the network model \eqref{eq:sysnet} is the factor $1/h^2$ in front of the diffusive terms coming from the finite difference scheme. Roughly speaking, it allows the eigenvalues of the graph Laplacian to approximate the first $N$ eigenvalues of the Laplace operator when $h$ becomes small (namely increasing the number of mesh points/nodes).
\end{document}